\documentclass[a4paper,11pt,reqno]{amsart}

\pdfoutput=1

\usepackage[utf8]{inputenc}
\usepackage[T1]{fontenc}
\usepackage{lmodern}
\usepackage[french,english]{babel}
\usepackage{amsthm, amssymb, amsmath, amsfonts, mathrsfs, dsfont, esint,textcomp}
\usepackage[colorlinks=true, pdfstartview=FitV, linkcolor=blue, citecolor=blue, urlcolor=blue,pagebackref=false]{hyperref}
\usepackage{mathtools}
\usepackage{tikz}
\usepackage{enumitem, imakeidx}
\usetikzlibrary{patterns}
\definecolor{asparagus}{rgb}{0.53, 0.66, 0.42}
\usepackage{a4wide}  
\usepackage{color}
\usepackage{microtype}
\newtheorem{thm}{Theorem}[section]

\newtheorem{lem}[thm]{Lemma}

\theoremstyle{remark}
\newtheorem{rem}[thm]{Remark}
\theoremstyle{definition}

\renewcommand{\leq}{\leqslant}
\renewcommand{\geq}{\geqslant}

\renewcommand{\subset}{\subseteq}

\newcommand{\E}{\mathbb{E}}

\newcommand{\N}{\mathbb{N}}

\newcommand{\1}{\mathbf{1}}
\newcommand{\R}{\mathbb{R}}

\newcommand{\Z}{\mathbb{Z}}

\renewcommand{\P}{\mathbb{P}}

\newcommand{\eps}{\varepsilon}
\renewcommand{\d}{{\mathrm{d}}}

\newcommand{\red}{\color{red}}

\newcommand{\aeps}{\eps^\frac{d}{d-2}}

\DeclareMathOperator{\capacity}{Cap}

\mathtoolsset{showonlyrefs}

\newcommand{\avsum}{\mathop{\mathpalette\avsuminner\relax}\displaylimits}

\makeatletter
\newcommand\avsuminner[2]{%
  {\sbox0{$\m@th#1\sum$}%
   \vphantom{\usebox0}%
   \ooalign{%
     \hidewidth
     \smash{\vrule height\dimexpr\ht0+1pt\relax depth\dimexpr\dp0+1pt\relax}%
     \hidewidth\cr
     $\m@th#1\sum$\cr
   }%
  }%
}
\makeatother

\title{Convergence rates for the homogenization of the Poisson problem in randomly perforated domains}
                        
\begin{document}
\maketitle
\smallskip
\begin{center}

\bigskip

{\sc Arianna Giunti}
\end{center}

\bigskip

{\bf Abstract.} In this paper we provide converge rates for the homogenization of the Poisson problem with Dirichlet boundary conditions in a randomly perforated domain of $\R^d$, $d \geq 3$. We assume that the holes that perforate the domain are spherical and are generated by a rescaled marked point process $(\Phi, \mathcal{R})$. The point process $\Phi$ generating the centres of the holes is either a Poisson point process or the lattice $\Z^d$; the marks $\mathcal{R}$ generating the radii are unbounded i.i.d random variables having finite $(d-2+\beta)$-moment, for $\beta > 0$. We study the rate of convergence to the homogenized solution in terms of the parameter $\beta$. We stress that, for certain values of $\beta$, the balls generating the holes may overlap with overwhelming probability.

\section{Introduction}

In this paper we obtain convergence rates for the homogenization of the Poisson problem in a bounded domain of $\R^d$, $d \geq 3$, that is perforated by many small random holes $H^\eps$. We impose with Dirichlet boundary conditions on the boundary of the set and of the holes $H^\eps$. We assume that, for $\eps > 0$, the random set $H^\eps$ is generated by a rescaled \textit{marked point process $(\Phi, \mathcal{R})$}, where $\Phi$  is either the lattice $\Z^d$ or a Poisson point process of intensity $\lambda > 0$. The associated marks $\mathcal{R}= \{ \rho_z \}_{z \in \Phi}$ are independent and identically distributed random variables that satisfy the moment condition
\begin{align}\label{integrability.radii.intro}
\E \bigl[ \rho^{d-2+\beta} \bigr] < +\infty, \ \ \ \beta > 0.
\end{align}
More precisely, given $(\Phi, \mathcal{R})$ and a bounded and smooth domain $D\subset \R^d$, we define
\begin{align}\label{def.holes}
H^\eps:= \bigcup_{z \in \Phi \cap (\frac{1}{\eps}D)} B_{(\aeps \rho_{z}) \wedge 1} (\eps z), \ \ \ \ \ \  D^\eps:= D \backslash H^\eps
\end{align}
with $(\frac{1}{\eps}D):= \{ x \in \R^d \, \colon \, \eps x \in D \}$. As shown in \cite{GHV}, if $\beta=0$ in \eqref{integrability.radii.intro}, then for every $f \in H^{-1}(D)$  and $\P$-almost every realization of the random set $H^\eps$, the solutions to 
\begin{align}\label{P.eps}
\begin{cases}
-\Delta u_\eps = f \ \ \ \ &\text{in $D^\eps$}\\
u_\eps = 0 \ \ \ \ &\text{on $\partial D^\eps$}
\end{cases}
\end{align}
converge weakly in $H^1_0(D)$ to the homogenized problem
\begin{align}\label{P.hom}
\begin{cases}
-\Delta u + C_0 u = f \ \ \ \ &\text{in $D$}\\
u = 0 \ \ \ \ &\text{on $\partial D$.}
\end{cases}
\end{align}
The constant $C_0 > 0$ is the limit of the density of harmonic capacity generated by the set $H^\eps$: If $S^{d-1}$ denotes the $(d-1)$-dimensional unit sphere, then
\begin{equation}\label{strange.term}
C_0 := c_d\begin{cases}
\E_\rho \bigl[ \rho^{d-2} \bigr] \ \ \ \ & \text{if $\Phi = \Z^d$}\\
 \lambda \E_\rho \bigl[ \rho^{d-2} \bigr] & \text{if $\Phi = \mathop{Poi}(\lambda)$}
\end{cases}, \ \ \ \ c_d:= (d-2) \mathcal{H}^{d-1}(S^{d-1}).
\end{equation}
In this paper, we strengthen the condition of \cite{GHV} from $\beta=0$ to $\beta > 0$ in \eqref{integrability.radii.intro} and study the convergence rates of $u_\eps$ to the homogenized solution $u$.

\bigskip

By the Strong Law of Large Numbers, assumption \eqref{integrability.radii.intro} with $\beta=0$ is minimal in order to ensure that for $\P$-almost every realization of $H^\eps$, its density of capacity admits a finite limit. However, it does not prevent the balls in $H^\eps$ from having radii that are much bigger than size $\aeps$. This gives rise to clustering phenomena with overwhelming probability. In particular, for $\beta < d -2$, the expected number of balls of $H^\eps$ that intersect, namely such that their radius $\aeps \rho_z$ is bigger than the typical distance $\eps$ between the centres, is of order $\eps^{-d+2+\beta}$ (over an expected total of $\eps^{-d}$ balls). The same holds also under assumption \eqref{integrability.radii.intro} for  $\beta < \frac{(d-2)^2}{2}$, with the expected number of overlapping balls being of order $\eps^{-d +2 +\frac{2}{d-2}\beta}$ .
The presence of balls that overlap is the main challenge in the proof of the qualitative homogenization statement obtained in \cite{GHV} and is one of the challenges of the current paper. It requires a careful treatment of the set $H^\eps$ to ensure that the presence of long chains of overlapping balls does not destroy the homogenization process. For a more detailed discussion on this issue we refer to the introductory section in \cite{GHV} and to Subsection \ref{sub.ideas} of the present paper.

\bigskip

The main results contained in this paper provide an annealed (i.e. averaged in probability) estimate for the $H^1$-norm of the homogenization error $u_\eps- W_\eps u$. The function $W_\eps$ is a suitable corrector function that is related to the so-called \textit{oscillating test function} \cite{Cioranescu_Murat,Tartar}. We assume that $\Phi$ is the lattice $\Z^d$ or that it is a Poisson point process in dimension $d=3$. If $\E \bigl[ \cdot \bigr]$ denotes the expectation under the probability measure associated to the process $(\Phi, \mathcal{R})$, we show that\footnote{In the case of $\Phi$ being a Poisson point process, there is a factor $\log\eps$ on the right-hand side. We refer to Theorem \ref{t.main} for the precise statement.}
\begin{align}\label{thm.rough}
\E\bigl[ \| u_\eps - W_\eps u \|_{H^1_0(D)}^2 \bigr]^{\frac 1 2} \leq C\begin{cases}
\eps^{\frac{d}{d^2- 4}\beta} \ \ \ &\text{if $\beta \leq d-2$}\\
\eps^{\frac{d}{d+2}} \ \ \ &\text{if $\beta > d-2$}
\end{cases}
\end{align}
We stress that in the case of periodic holes, namely when $\Phi=\Z^d$ and $\rho_z \equiv r > 0$ for all $z \in \Z^d$, the optimal rate on the right-hand side of \eqref{thm.rough} is $\eps$ \cite{Kacimi_Murat}. 

\bigskip

The main quantity that governs the decay of the homogenization error $u_\eps - W_\eps u$ is the convergence of the capacity density of $H^\eps$ to the constant term $C_0$ defined in \eqref{strange.term}. In the periodic case mentioned in the previous paragraph, the term $C_0= c_d r^{d-2}$ is ``very'' close to the density of capacity of $H^\eps$ already at scale $\eps$. Heuristically, indeed, if $A \subset D$ we have
\begin{align}\label{density.capacity}
\capacity(A \cap H^\eps) \simeq \sum_{z \in (\eps Z)^d \cap A} \capacity(B_{\aeps r}(z)) \simeq |A| \eps^{-d}  c_d (\aeps r)^{d-2} \stackrel{\eqref{strange.term}}{=} C_0 |A|,
\end{align}
and this chain of identities is true as long as $|A|$ is at least of order $\eps$. On the other hand, in our setting, this identity is expected to hold at scales that are larger than $\eps$ due to the fluctuations of the process $(\Phi, \mathcal{R})$. For a more detailed explanation of the exponents in \eqref{thm.rough}, we refer to Subsection \ref{sub.ideas}. {We also remark that the threshold $d-2$ in the parameter $\beta$ obtained in \eqref{thm.rough} is related to the $L^2$-nature of the norm considered for the homogenization error. Roughly speaking, the norm considered in \eqref{thm.rough} requires a control on the expectation of the square of the capacity generated by the balls in $H^\eps$.}

\bigskip

Starting with \cite{Cioranescu_Murat} and \cite{papvar.tinyholes}, there is a large amount of literature devoted to the homogenization of \eqref{P.eps}, both for deterministic \cite{DalMasoGarroni.punctured, HoeferVelazquez.reflections} and random holes $H^\eps$ \cite{CaffarelliMellet, CasadoDiaz, MarchenkoKhruslov}; similar problems have also been studied in the case of the fractional laplacian $(-\Delta)^s$, \cite{CaffarelliMellet.fractional, Focardi.fractional} or for nonlinear elliptic operators \cite{CasadoDiaz.nonlinear, Zhikov}. All the models considered in the deterministic case contain assumptions that ensure that,  for $\eps$ small enough, the holes in $H^\eps$  do not to overlap. In the random models mentioned above, the previous property is as well required, at least for $\P$-almost every realization and $\eps>0$ small enough. For a complete and more detailed description of these works, we refer to the introduction of \cite{GHV}. 

\smallskip

We also mention that the analogue of \eqref{P.eps} for a Stokes (and Navier-Stokes) system with no-slip boundary conditions on the holes $H^\eps$ has been considered in \cite{AllaireARMA1990a, Allaire_arma2, SanchezP82} in the periodic case and then extended to more general configurations of holes (see, e.g., \cite{Desvillettes2008, Hillairet2018, Hillairet2017}). In the case of the Stokes operator, the limit equation contains an additional zero-th order term similar to $C_0$ in \eqref{P.hom}. Under the same assumptions of this paper, the analogue of the homogenization result contained in \cite{GHV} has been proven for a Stokes system in \cite{GH1, GH_pressure}. 

\smallskip

In the periodic case, quantitative rates of convergence for \eqref{P.eps} to \eqref{P.hom} have been first obtained in \cite{Kacimi_Murat}. In \cite{Russel, Wang_Xu_Zhang}, similar results have been obtained with $-\Delta$ replaced by (also nonlinear) oscillating elliptic operators. When the holes are randomly distributed, the first quantitative result on the convergence of $u_\eps$ to $u$ has been obtained in \cite{Figari_Orlandi}. In this paper, the authors study the analogue of \ref{P.eps} for the operator $-\Delta + \lambda$ in an unbounded domain of $\R^3$, that is perforated by $m$ spherical holes of identical radius $\sim m^{-1}$. The centres of the holes are independent and distributed according to a compactly supported and continuous potential $V$. If $u_m$ denotes the analogue of $u_\eps$,when the massive term $\lambda$ is big enough (compared to $V$), the authors provide rates of convergence for the $L^2$-norm of the difference $u_m -u$ in the limit $m \to +\infty$. Furthermore, they prove the Gaussianity of the fluctuations of $u_m$ around the homogenized solution $u$ in the CLT-scaling. In \cite{Jonas_Richard}, this result has been obtained in the same setting of \cite{Figari_Orlandi} without any constraint on the massive term $\lambda$. 
 
\bigskip

\subsection*{Organization of the paper.} This paper is organized as follows: In Section \ref{s.main} we introduce the setting and state the main results. In Subsection \ref{sub.ideas}, we provide an overview of the main challenges and ideas used to prove Theorem \ref{t.main} . In Section \ref{s.periodic} we prove Theorem \ref{t.main}, $(a)$ while in Section \ref{s.poi.d} we show how to extend the argument of the previous section when $\Phi$ is a Poisson point process in $\R^3$ (Theorem \ref{t.main}, case $(b)$). Finally, Section \ref{Aux} contains some auxiliary results that are used in the proofs of the main results.

\section{Setting and main results}\label{s.main}
Let $d \geq 3$ and $D \subset \R^d$ be a bounded and smooth domain that is star-shaped with respect to the origin. For $\eps > 0$, we define the random set of holes $H^\eps$ and the punctured set $D^\eps$ as in \eqref{def.holes}.

\smallskip

We assume that the union of balls $H^\eps$ is generated by a marked point process $(\Phi, \mathcal{R})$ on $\R^d \times \R_+$. In other words, we generate the centres of the balls in $H^\eps$ via a point process $\Phi$. To each point $z\in \Phi$, we associate a mark $\rho_z \geq 0$ that determines the radius of the ball. We refer to \cite[Chapter 9, Definitions 9.1.I - 9.1.IV]{Daley.Jones.book2}
 for an extensive and rigorous definition of marked point processes and their associated measures on $\R^d \times \R_+$.  We denote by $(\Omega; \mathcal{F}, \mathbb{P})$ the probability space associated to $(\Phi, \mathcal{R})$, so that the random sets in \eqref{def.holes} and the random field solving \eqref{P.eps} may be written as $H^\eps= H^\eps(\omega)$, $D^\eps=D^\eps(\omega)$ and $u_\eps(\omega; \cdot)$, respectively. The set of realizations $\Omega$ may be seen as the set of atomic measures $\sum_{n \in \N} \delta_{(z_n, \rho_n)}$ in $\R^d \times \R_+$ or, equivalently, as the set of (unordered) collections $\{ (z_n , \rho_n \}_{ n\in \N} \subset \R^d \times \R_+$. 

\smallskip

Throughout this paper we assume that $(\Phi, \mathcal{R})$  satisfies the following conditions:
\begin{itemize}
\item[(i)] $\Phi$ is either the lattice $\Z^d$ or  $\Phi = \mathop{Poi}(\lambda)$, i.e. a Poisson point process of intensity $\lambda>0$;

\smallskip

\item[(ii)] The marks $\{ \rho_z\}_{z \in \Phi}$ are independent and identically distributed: if $\P_{\mathcal{R}}$ denotes the marginal of the marks with respect to the process $\Phi$, then the n-correlation function may be written as the product
\begin{align}
f_n ( (z_1, \rho_1), \cdots, (z_n, \rho_n)) = \Pi_{i=1}^n f_1((z_i, \rho_i)), \ \ \ \ f_1((z, \rho)) = f(\rho).
\end{align}

\smallskip

\item[(iii)] The marks $\mathcal{R}$ have finite $(d-2+\beta)$-moment, namely the density function $f$ in (ii) satisfies
\begin{align}\label{integrability.radii}
\E_{\rho}\bigl[ \rho^{d-2+\beta}\bigr] : = \int_0^{+\infty} \rho^{d-2 +\beta} f(\rho) \d \rho \leq 1, \ \ \ \ \text{with $\beta > 0$.}
\end{align}
\end{itemize}

We stress that conditions (i)-(ii) yield that $(\Phi, \mathcal{R})$ is stationary. In the case $\Phi=\mathop{Poi}(\lambda)$, the process $(\Phi, \R)$ is stationary with respect to the action of the group of translations $\{\tau_x \}_{x\in \R^d}$. This means that the probability measure $\P$ is invariant under the action of the transformation $\tau_x: \Omega \to \Omega$, \, $\omega=  \{ ( z_i; \rho_{z_i}) \}_{i \in \N} \mapsto \tau_x \omega := \{ ( z_i + x; \rho_{z_i}) \}_{i \in \N}$. In the case $\Phi=\Z^d$ the same holds under the action of the group $\{\tau_z \}_{z\in \Z^d}$.

\bigskip

\subsection*{Notation.} When no ambiguity occurs, we skip the argument $\omega \in \Omega$ in the notation for  $H^\eps(\omega), D^\eps(\omega)$ and $u_\eps(\omega ; \cdot)$, as well as in all the other random objects. We denote by $\E\bigl[ \cdot \bigr]$ and $\E_{\Phi}\bigl[ \cdot \bigr]$ the expectations under the total probability measure $\mathbb{P}$ the probability measure $\mathbb{P}_{\Phi}$ associated to the point process $\Phi$.  For $\eps > 0$ and a set $A \subset \R^d$, we define 
\begin{align}\label{notation.psi}
\Phi(A):=\bigl\{ z \in \Phi \, \colon \, z \in A \bigr\}, \ \ \  \Phi_\eps(A) := \bigl\{ z \in \Phi \, \colon \, \eps z \in A \bigr\}
\end{align}
and the random variables
\begin{align}
N(A):= \#(\Phi(A)), \ \ \ N^\eps(A):=  \#(\Phi^\eps(A)).
\end{align}

\smallskip

For any $\mu \in H^{-1}(D)$, we write $\langle \, \cdot \, ;  \, \cdot \, \rangle$ for the duality product with $H^1_0(D)$; we use the notation $\avsum_{i \in I}$ for the averaged sum $\#(I)^{-1}\sum_{i\in I}$ and $\lesssim$ and $\gtrsim$ instead of $\leq C$ and $\geq C$ with the constant $C$ depending on the dimension $d$, the domain $D$ and, in the case of $\Phi=\mathop{Poi}(\lambda)$, the intensity rate $\lambda$.

\bigskip

\subsection{Main result} 
Before stating the main results, we need to define a suitable corrector function $W_\eps$ that appears in the homogenization error $u_\eps - W_\eps u$.
 We stress that, also in the case of periodic holes, the solutions $u_\eps$ are only expected to converge weakly in $H^1_0(D)$ to $u$. Therefore, the homogenized solution needs to be suitably modified via a corrector $W_\eps$ in order to be a good approximation for $u_\eps$ also { in the strong topology of $H^1_0(D)$}. 

\smallskip

For $x \in \R^d$ we set
\begin{align}\label{minimal.distance}
R_{\eps,x} := \frac \eps 4 \min_{z \in \Phi^\eps(D), \atop  z \neq x} \biggl\{|z - x| ; 1 \biggr\}
\end{align}
Note that, if $\Phi= \Z^d$, then the above quantity is always $\frac \eps 4$. For $\delta > 0$, we denote by $\Phi_{\delta}^\eps(D) \subset \Phi_\eps(D)$ the set
\begin{align}\label{thinning.psi}
\Phi_{\delta}^{\eps}(D):= \biggl\{ z \in \Phi_\eps(D) \, \colon  \, \aeps\rho_z \leq \eps^{1 +\delta}, \, \, R_{\eps,z} \geq 2 \sqrt d \aeps\rho_z \biggr\}.
\end{align}

\smallskip

For each $z \in \tilde \Phi^\eps(D)$, let $w_{\eps,z} \in H^1(B_{\frac \eps 2}(\eps z))$ be the solution to
\begin{align}\label{def.harmonic.annuli}
\begin{cases}
-\Delta w_{z,\eps} = 0 \ \ \ \ &\text{in $B_{R_{\eps,z}}(\eps z)\backslash B_{\aeps\rho_z}(\eps z)$}\\
w_{z,\eps} = 0 \ \ \ &\text{on $\partial B_{R_{\eps,z}}(\eps z)$}\\
w_{z,\eps} = 1 \ \ \ &\text{on $\partial B_{R_{\eps,z}}(\eps z)$}.
\end{cases}
\end{align}
We thus define
\begin{align}\label{corrector}
W_\eps(x)=\begin{cases}
w_{z,\eps} \ \ \ &\text{if $x\in B_{R_{\eps,z}}(\eps z)\backslash B_{\aeps\rho_z}(\eps z)$}\\
0 \ \ \ &\text{if $x \in B_{\aeps\rho_z}(\eps z)$}\\
1 \ \ \ &\text{otherwise}
\end{cases}
\end{align}
We stress that \eqref{thinning.psi} ensures that definitions \eqref{def.harmonic.annuli} and \eqref{corrector} are well-posed since the set $\{ B_{R_{\eps,z}}(\eps z)\}_{z \in \Phi_\delta^\eps(D)}$ is made of disjoint balls and, for every $z \in \Phi^\eps_\delta(D)$, it holds $B_{\aeps \rho_z}(\eps z) \subset B_{R_{\eps,z}}(\eps z)$. Note that in the above definition the function $W_\eps \in H^1(D)$ depends on the choice of the parameter $\delta$ used to select the subset $\Phi_\delta^\eps(D)$. The optimal parameter $\delta$ will be fixed in Theorem \ref{t.main}. We finally stress that, in the periodic case $\Phi= \Z^d$ and $\rho_z \equiv  r$, for any $\delta > 0$ and $\eps$ small enough, the function $W_\eps$ coincides with the \textit{oscillating test function} constructed in \cite{Cioranescu_Murat, Kacimi_Murat}. 

\bigskip

\begin{thm}\label{t.main}
Let $(\Phi, \mathcal{R})$ satisfy conditions (i)-(iii) of the previous subsection. For $\eps>0$ and $f \in L^\infty(D)$, with $\|f\|_{L^\infty(D)}=1$, let $u_\eps$ and $u$ be as in \eqref{P.eps} and \eqref{P.hom}, respectively. We consider the random field $W_\eps$ in \eqref{corrector} with 
\begin{align*}
\delta = \begin{cases}
\frac{4}{d^2- 4} \ \ \ &\text{if $\beta \leq d-2$}\\
\frac{2}{d-2}- \frac{2d}{(d+2)\beta} \ \ \ &\text{if $\beta > d-2$}
\end{cases}
\end{align*}.
Then
\begin{itemize}
\item[(a)] If $\Phi= \Z^d$, there exists a constant $C= C(d, D)> 0$ such that
\begin{align*}
\E \bigl[ \| u_\eps - W_\eps u \|_{H^1_0(D)}^2 \bigr]^{\frac 1 2} \leq C\begin{cases}
\eps^{\frac{d}{d^2- 4}\beta} \ \ \ &\text{if $\beta \leq d-2$}\\
\eps^{\frac{d}{d+2}} \ \ \ &\text{if $\beta > d-2$}
\end{cases}
\end{align*}

\item[(b)] If $\Phi = \mathop{Poi}(\lambda)$ with $\lambda >0$ and $d=3$, there exists a constant $C=C(\lambda, D)>0$ such that
\begin{align*}
\E \bigl[ \| u_\eps - \tilde W_\eps u \|_{H^1_0(D)}^2 \bigr]^{\frac 1 2} \leq C\begin{cases}
|\log \eps|\eps^{\frac{3}{5}\beta} \ \ \ &\text{if $\beta \leq 1$}\\
|\log \eps| \eps^{\frac{3}{5}} \ \ \ &\text{if $\beta > 1$}
\end{cases}
\end{align*}

\end{itemize}

\end{thm}

\bigskip

\begin{rem}
As it becomes apparent in the proof of Theorem \ref{t.main}, the choice of $W_\eps$ is not unique. The same result holds, for instance, if $W_\eps$ is replaced with the oscillating test function $w_\eps$ constructed in {\cite[Section 3]{GHV}} and in Subsection \ref{sub.quenched.periodic} of the present paper. The function $W_\eps$, however, has a simpler and more explicit construction that may be implemented numerically with more efficiency. It is, indeed, an oscillating test function restricted to the balls of $H^\eps$ that do not overlap and have radius smaller than the fixed threshold $\eps^{1+\delta}$.
\end{rem}

\bigskip

\subsection{Ideas of the proofs}\label{sub.ideas}

The proof of Theorem \ref{t.main} is inspired to the proof of the same result in the case of periodic holes shown in \cite{Kacimi_Murat}. The latter, in turn, upgrades the result of \cite{Cioranescu_Murat} from the qualitative statement $u_\eps \rightharpoonup u$ in $H^1_0(D)$ to an estimate on the convergence of the homogenization error. Both arguments rely on the construction of suitable \textit{oscillating test functions} $\{ w_\eps \}_{\eps> 0} \subset H^1(D)$. In the qualitative statement of \cite{Cioranescu_Murat}, these functions allow to pass to the limit $\eps \downarrow 0$ in the weak formulation of \eqref{P.eps} and infer the homogenized equation of \eqref{P.hom}. 

\smallskip

The functions $\{ w_\eps \}_{\eps >0}$ may be constructed as $W_\eps$ in \eqref{corrector}, where the set $\Phi^\eps_\delta$ coincides with the whole set $\Phi=\Z^d$ and $w_{\eps,z}= w_{\eps,0}( \cdot + z)$. Furthermore, they are strictly related to the density of capacity generated by $H^\eps$:  The additional term $C_0= c_d r^{d-2}$ that appears in the homogenized equation \eqref{P.hom} is indeed the limit of the measures $-\Delta w_\eps$ when tested against the function $\rho u_\eps \in H^1_0(D^\eps)$, $\rho \in C^\infty_0(D)$. It is not hard to see from \eqref{corrector} that, for functions that vanish on the holes $H^\eps$, the action of $-\Delta w_\eps$ reduces the periodic measure 
\begin{align}\label{mu.eps}
\mu_\eps = \sum_{z \in \Z^d \cap \frac{1}{\eps}D} \partial_n w_{\eps,z} \delta_{ \partial B_{\frac \eps 4}(\eps z) },
\end{align}
that is concentrated on the spheres $\{ \partial B_{\frac \eps 4}(\eps z) \}_{z \in \Z^d}$.

\smallskip

In \cite{Kacimi_Murat}, the corrector $W_\eps$ is chosen as the oscillating test function $w_\eps$ itself. As a first step, it is shown that the decay of $\| u_\eps - W_\eps u \|_{H^1_0(D)}$ boils down to controlling the convergence of the density of capacity of $H^\eps$ to its limit $C_0$ (c.f. \eqref{strange.term}). The latter is expressed in terms of the decay of the norm $\|\mu_\eps- C_0\|_{H^{-1}(D)}$. As a second step, the authors appeal to a result of \cite{Kohn_Vogelius} to estimate the decay of  $ \| \mu_\eps- C_0\1_D \|_{H^{-1}(D)}$ in terms of the size $\eps$ of the periodic cell $C_\eps:=[-\frac \eps 2; \frac \eps 2]$ of $\mu_\eps$. The crucial feature is that, up to a correction of order $\eps^2$, the measure $\mu_\eps -C_0$ has zero average in $C_\eps$. In other words, we have
\begin{align}\label{zero.average}
\int_{\partial B_{\frac \eps 4}(0)} \partial_n w_{\eps} =  \eps^d \bigl( C_0 + O(\eps^2)\bigr).
\end{align}

\bigskip

In this paper we adapt to the random setting the previous two-step argument. The first main difference is strictly related to the randomness of the radii in $H^\eps$ and needs to be addressed also in the case of bounded radii (i.e. if $\beta= +\infty$ in \eqref{integrability.radii}) and periodic centres.  In this case, the measure $\mu_\eps$ is defined as in \eqref{mu.eps} but, on each sphere $\partial B_{\frac \eps 4}(\eps z)$, $z \in \Z^d$, the term $\partial_n w_\eps$ depends on the random associated mark $\rho_z$. Therefore, contrarily to the periodic case, \eqref{zero.average} may not hold in each cube $\eps z + C_\eps$. Nevertheless, by the Law of Large Numbers, we may expect that the average of $\mu_\eps - C_0$ is close to zero over cubes of size $k\eps$, $k >>1$, as the left-hand side in \eqref{zero.average} turns into an averaged sum of $k$ random variables. This motivates the introduction
of a partition of the set $D$ into cubes of mesoscopic size $k\eps$ (c.f. Section \ref{sub.covering}) that plays the role of the cells $C_\eps + \eps z$ of the periodic case. This allows us to adapt the result by \cite{Kohn_Vogelius} and obtain
\begin{align}\label{estimate.scaling}
\E\bigl[ \| \mu_\eps- C_0\1_D \|_{H^{-1}(D)}^2 \bigr]^{\frac 1 2} \lesssim k \eps + \E_\rho\biggl[ (\avsum_{i=1}^k \rho_i^{d-2} - \E_\rho \bigl[ \rho^{d-2} \bigr])^2 \biggr].
\end{align}
Here, the last term accounts for the difference between the average of $\mu_\eps$ in each cube of size $(k\eps)$, $k \in \N$ and the value $C_0$. This inequality, implies an estimate of the form:
\begin{align}
\E\bigl[ \| \mu_\eps- C_0\1_D \|_{H^{-1}(D)}^2 \bigr]^{\frac 1 2} \lesssim k \eps + \E_\rho\biggl[(\rho^{d-2} - \E_\rho\bigl[ \rho^{d-2} \bigr])^2 \biggr]k^{-\frac d 2}.
\end{align}
The optimal choice of $k$ yields the exponent $\frac{d}{d+2}$ of Theorem \ref{t.main}. If $\rho_z \equiv r$ for all $z \in \Z^d$, then the second term vanishes and the above estimate with $k=1$ gives the optimal rate of \cite{Kacimi_Murat}.

\smallskip

In the case of centres distributed according to a Poisson point process, the argument for Theorem \ref{t.main} follows the same ideas above;  although the centres of the holes in $H^\eps$ have random positions, their typical distance is indeed still of size $\eps$. This feature gives rise to the additional logarithmic factor in the rate of Theorem \ref{t.main}. The main technical challenge is related to the construction of the mesoscopic partition of $D$ that allows to obtain the analogue of \eqref{estimate.scaling}. In contrast with the case $\Phi=\Z^d$, indeed, there are ($\P$-sufficiently many) realizations of $H^\eps$ where the support of the measure $\mu_\eps$ intersects the boundary of the covering. In other words, the spheres $\{ \partial B_{\frac \eps 4}(\eps z) \}_{z\in \Phi^\eps_\delta(D)}$ might fall across two cubes of size $\eps k$ that cover $D$. This, in particular, implies that to the covering does not correspond a well-defined partition of the spheres where the measure $\mu_\eps$ is supported. We tackle this issue by constructing a suitable random covering.  We do this by enlarging each cube of size $\eps k$ so that it also includes the spheres $B_{\frac \eps 4}(\eps z)$ that fall on its boundary (see also Figure \ref{covering.pic}). In order to obtain the wanted rate of convergence, we require that the new sets have volume that is ``very close''  to the one of deterministic partition into cubes that is used in the case $\Phi= \Z^d$. We do so by restricting the size of the spheres that are too close to the boundary from size $\eps$ to size $\eps^{1+\kappa}$, where $\kappa =\kappa(d)>0$ is a suitable exponent. We refer to Subsection \ref{sub.covering} for the precise construction.

\smallskip

A second challenge that arises in the proof of Theorem \ref{t.main} is related to the presence of overlapping holes in the case $\beta < +\infty$ in \eqref{integrability.radii}. The strategy to deal with this issue is very similar to the one used in \cite{GHV}: We construct, indeed, a suitable partition of $H^\eps= H^\eps_b \cup H^\eps_g$, where the subset $H^\eps_b$ contains all the holes that overlap (c.f. Lemma \ref{l.geometry.periodic}). As shown in \cite{GHV}, the contribution of $H^\eps_b$ to the density of capacity is negligible in the limit $\eps \downarrow 0$. As a consequence, we may modify the estimates of \cite{Kacimi_Murat}, to prove that we may control the decay of $\| u_\eps- W_\eps u \|_{H^1_0(D)}$ with the decay of the norm $ \| \mu_\eps- C_0\1_D \|_{H^{-1}(D)}$, where the measure $\mu_\eps$ is now only related to the union of disjoint balls $H^\eps_g$.

\smallskip

\smallskip

\section{Proof of Theorem \ref{t.main}, $(a)$}\label{s.periodic}

\subsection{Partition of the holes $H^\eps$ and mesoscopic covering of $D$}\label{sub.covering}
This section  contains some technical tools that will be crucial to prove the main result: The first one is an adaptation of \cite{GHV} and provides a suitable way of dividing the holes $H^\eps$ between the ones that may overlap due to the unboundedness of the marks $\{ \rho_z\}_{z \in \Phi}$ and the ones that, instead, are disjoint and have radii $\aeps \rho_z$ much smaller than the distance $\eps$ between the centres.

\bigskip

\begin{lem}\label{l.geometry.periodic}  Let $\delta \in (0, \frac{2}{d-2}]$ be fixed. There exists an $\eps_0=\eps_0(\delta,d)$ such that for every $\eps \leq \eps_0$ and $\omega \in \Omega$ we may find a partition of the realization of the holes
$$
H^\eps:= H^\varepsilon_{g} \cup H^\varepsilon_{b}
$$
with the following properties:
\begin{itemize}
\item There exists a subset of centres $n^\eps(D) \subset \Phi^\eps(D)$ such that
\begin{align}\label{good.set.periodic}
H^\varepsilon_g: = \bigcup_{ z \in n^\varepsilon(D)} B_{\aeps \rho_z}( \varepsilon z ), \ \ \max_{z \in n^\eps(D)}\aeps \rho_z \leq \eps^{1+\delta};
\end{align}

\smallskip

\item There exists a set $D^\eps_b \subset  \{ x \in \R^3 \, \colon \, \mathop{dist}(x, D) \leq 2\}$ satisfying
\begin{align}
H^\varepsilon_{b} \subset D^\eps_b, \ \ \ \capacity ( H^\varepsilon_b,D_b^\eps) \lesssim  \eps^d \sum_{z \in \Phi^\eps(D) \backslash n^\eps(D)} \rho_z^{d-2} \label{capacity.sum.periodic}
\end{align}
and
\begin{align}\label{distance.good.bad.periodic}
B_{\frac{\eps}{4}}(\eps z) \cap D^\eps_b = \emptyset, \ \ \  \ \ \ \ \text{for every $z \in n^\eps(D)$.} 
\end{align}

\end{itemize}
\end{lem}

\bigskip

\begin{proof}[Proof of Lemma \ref{l.geometry.periodic}]
The construction for the sets $H^\eps_g, H^\eps_b$ and $D^\eps_b$ is the one implemented in the proof of  \cite[Lemma 2.2]{GHV}. We fix $\delta \in (0, \frac{2}{d-2}]$ throughout the proof.

\smallskip

We denote by $I_\eps^b \subset \Phi_\eps(D)$ the set that generates the holes $H_b^\eps$. We construct it in the following way: We first consider the points $z \in \Phi^\eps(D)$ whose marks $\rho_z$ are bigger than $\eps^{-\frac{2}{d-2}+\delta}$, namely 
\begin{align}\label{index.set.J}
J^\eps_b= \Bigl\{ z \in \Phi^\eps(D) \colon \aeps\rho_z \geq \eps^{1+\delta}\Bigr\}.
\end{align}
Given the holes
\begin{align*}
\tilde H^\eps_b := \bigcup_{z \in  J^\eps_b} B_{2 (\aeps \rho_z \wedge 1)}(\eps z),
\end{align*}
we include in $I_b^\eps$ also the set of points in $\Phi_\eps(D) \backslash J^\eps_b$ that are ``too close'' to the set $\tilde H_b^\eps$, i.e.
\begin{align}\label{index.set.I.tilde}
 \tilde I^\eps_b := \left \{ z \in \Phi^\eps(D) \backslash J^\eps_b \colon  \tilde H^\eps_b  \cap B_{\frac{\eps}{4}}(\eps z) \neq \emptyset \right\}.
\end{align}
 We define
 \begin{equation}
\begin{aligned}\label{definition.bad.index}
	I^\eps_b := \tilde I^\eps_b \cup J^\eps_b, \ \ \ n^\eps &(D) := \Phi^\eps(D) \backslash I^\eps_b \\
H^\eps_b:=  \bigcup_{z \in I^\eps_b} B_{\aeps \rho_z \wedge 1}(\eps z), \ \ \ H^\eps_g:=  \bigcup_{z_j \in n^\eps(D)}& B_{\aeps \rho_z}(\eps z), \ \ \ D^\eps_b:= \bigcup_{z \in I^\eps_b}  B_{2(\aeps \rho_z \wedge 1)}(\eps z). 
\end{aligned}
\end{equation}

\smallskip

It remains to show that the sets defined above satisfy properties \eqref{good.set.periodic}-\eqref{distance.good.bad.periodic}. Property \eqref{good.set.periodic} is an immediate consequence of definition \eqref{index.set.J}. The first inclusion in \eqref{capacity.sum.periodic} easily follows by the definition of $H^\eps_b$ and $D^\eps_b$ in \eqref{definition.bad.index}; for the the inequality in \eqref{capacity.sum.periodic} we instead appeal to the subadditivity of the capacity to bound
\begin{align*}
\capacity(H^\eps_b; D^\eps_b) \leq \sum_{z \in \Phi^\eps(D) \backslash n^\eps(D)} \capacity(B_{\aeps\rho_z \wedge 1}(\eps z); D^\eps_b).
\end{align*}
Moreover, by the monotonicity property $\capacity(A; B) \leq \capacity(A; C)$ for every $B \supseteq C \supseteq A$, this turns into
\begin{align*}
\capacity(H^\eps_b; D^\eps_b) \leq \sum_{z \in \Phi^\eps(D) \backslash n^\eps(D)} \capacity(B_{\aeps\rho_z}(\eps z); B_{2\aeps\rho_z}(\eps z)) \lesssim  \eps^d\sum_{z \in \Phi^\eps(D) \backslash n^\eps(D)} \rho_z^{d-2},
\end{align*}
i.e. the estimate in \eqref{capacity.sum.periodic}.

\smallskip

To conclude the proof of this lemma, it remains to argue \eqref{distance.good.bad.periodic}: By construction (see \eqref{definition.bad.index}), it holds that
\begin{align}\label{bad.set.periodic.1}
D^\eps_b = \tilde H_b^\eps \cup \bigcup_{z \in \tilde I^\eps_b}B_{2\aeps \rho_z}(\eps z).
\end{align}
On the one hand, by the definition of $n^\eps(D)$ in \eqref{definition.bad.index} and \eqref{index.set.I.tilde}, for each $z \in n^\eps(D)$ we have that
\begin{align}\label{distance.periodic.1}
\mathop{dist}(\eps z; \tilde H_b^\eps) \geq \frac{\eps}{4}.
\end{align} 
On the other hand, again by \eqref{index.set.J}-\eqref{index.set.I.tilde}, if $w \in \tilde I^\eps_b$,  then $4\aeps \rho_w \leq \eps^{1+\delta}$ so that
\begin{align*}
\mathop{dist}(\eps z; B_{2 \aeps \rho_w}(\eps w)) \geq \frac \eps 2 |z -w| \geq \frac \eps 4,
\end{align*}
whenever $\eps$ is such that $\eps^\delta < \frac 1 4$. Hence, also
\begin{align*}
\mathop{dist}(\eps z;  \bigcup_{z \in \tilde I^\eps_b}B_{2\aeps \rho_z}(\eps z)) \geq \frac{\eps}{4}.
\end{align*} 
Combining this with \eqref{distance.periodic.1} and \eqref{bad.set.periodic.1}, we infer \eqref{distance.good.bad.periodic}. The proof of Lemma \ref{l.geometry.periodic} is complete.
\end{proof}

\bigskip

We now construct a suitable covering of $D$ that, as explained in Subsection \ref{sub.ideas}, plays a fundamental role in the proof of Theorem \ref{t.main}. We recall that, by our assumption, the set $D$ is any smooth domain that is star-shaped with respect to the origin. 

\smallskip

For $k \in \N$ and $z \in \Z^d$, let
\begin{equation}\label{cubes.meso}
Q_{k,z}:= \eps z +\frac {k\eps}{2} Q, \ \ \ Q := [-1 ; 1]^d.
\end{equation}
Let $N_k \subset \Z^d$ be such that the collection $\{ Q_{k, z} \}_{z \in N_k}$ is an essentially disjoint covering of $D$. Since $D$ is bounded, we may assume that 
\begin{align}\label{number.N.k}
\# (N_k) \lesssim (\eps k)^{-d}.
\end{align}
Let
\begin{align}\label{interior.cubes}
\mathring{N}_{k}:= \biggl\{ z \in N_k \, \colon \, Q_{k, z} \subseteq D, \, \mathop{dist}(Q_{k,z}; \partial D) \geq \eps \biggr\}.
\end{align}
Since $D$ is smooth and has compact boundary, it is easy to see that there exist $C_1=C_1(D)$ such that, whenever $k\eps \leq C$ it holds
\begin{align}\label{number.boundary.cubes}
\#({N}_{k} \backslash \mathring{N}_k) \lesssim (k\eps)^{d-1}.  
\end{align}
Finally, for each $z \in N_k$ we denote by $N_{k,z} \subset \Phi$ the set of points of $\Phi_\delta^\eps(D)$ that, when rescaled, are contained into the cube $Q_{\eps,z}$, i.e. such that
\begin{align}\label{number.covering}
N_{k, z}:= \{ w \in \Phi^\eps_\delta(D) \, \colon \, \eps w \in Q_{k,z} \} \stackrel{\eqref{notation.psi}}{=} \Phi^\eps_\delta(D) \cap \Phi^\eps(Q_{\eps,z})
\end{align}
Note that, since in this section we assumed that $\Phi=\Z^d$, it follows that
$$
\bigcup_{w \in N_{k,z}} Q_{\eps, w} \subset Q_{k,z}, \ \ \ \text{for every $z \in N_k$}
$$
and that, for every $z \in \mathring{N}_{k,z}$, the sets $\{Q_{\eps,w}\}_{w \in N_{k,z}}$ provide a refining of $Q_{k,z}$.

\bigskip

\subsection{Quenched estimates for the homogenization error}\label{sub.quenched.periodic}
All the results contained in this subsection are quenched, in the sense that they hold for any fixed realization of the holes $H^\eps$.  The main result of this section is Lemma \ref{det.estimate} that allows to control the norm of the homogenization error $u_\eps- W_\eps u$ in terms of suitable averaged sums of the random marks $\{\rho_z \}_{z \in \Phi}$. Lemma \ref{det.estimate} relies on Lemma \ref{det.KM}, that is an adaptation of \cite{Kacimi_Murat}[Theorem 3.2] and shows that controlling the error $u_\eps- W_\eps u$ considered in Theorem \ref{t.main} boils down to controlling the convergence to $C_0$ of the density of capacity generated by $H^\eps$. This, in turn, may be controlled using the mesoscopic covering $\{ Q_{k,z} \}_{z \in N_k}$ of the previous subsection with Lemma \ref{Kohn_Vogelius} of Section \ref{Aux}.

\bigskip

Before giving the statement of the first lemma, we recall the construction of the oscillating test function $w_\eps \in H^1(D)$ implemented in \cite{GHV}. As mentioned in the introduction and in Subsection \ref{sub.ideas}, the main feature of this function is to vanish on the holes $H^\eps$ and ``approximate'' the density of  the capacity of $H^\eps$. We note that the unboundedness of the marks $\{ \rho_z\}_{z\in \Phi}$ implies that the set $\Phi^\eps_\delta(D) \subsetneq \Phi^\eps(D)$  and that the function $W_\eps$ in \eqref{corrector} does not vanish in all the holes contained in $H^\eps$. 

\smallskip

Let $H^\eps_g, H^\eps_b$ and $D^\eps_b$ be as in Lemma \ref{l.geometry.ppp}. For every $z \in \Phi^\eps(D)$, we set
\begin{align}
v_\eps:= \mathop{argmin}\{  \int_{D_b^\eps} |\nabla u |^2 \ \colon \, u \in H^1_0(D_b^\eps), \  u = 1 \ \ \text{on $\partial H_b^\eps$} \},
\end{align}
i.e. the minimizer of $\capacity( H^\eps_b; D^\eps_b)$.\footnote{We assume that the minimizer exists. If this is not the case, then it suffice to take any function $v_\eps$ in the minimizing class such that $ \int_{D_b^\eps} |\nabla v_\eps |^2 \leq 2 \capacity( H^\eps_b; D^\eps_b)$.} 

\smallskip

We set as oscillating test function  
\begin{align}\label{oscillating.function}
w_\eps = w_\eps^g \wedge w_\eps^b
\end{align}
where $w_\eps^g$ and $w_\eps^b$ are defined as follows: 
\begin{align}\label{oscillating.bad}
w_\eps^b &:=\begin{cases} 1 - v_\eps \ &\text{in $D^\eps_b \backslash H^\eps_b$}\\
0 \ &\text{in $H^\eps_b$}\\
1 \ &\text{in $\R^3 \backslash D^\eps_b$}
\end{cases}
\end{align}
and 
\begin{align}\label{oscillating.good}
w_\eps^g(x) :=\begin{cases} w_{z,\eps} \ \ \ &\text{if $ x \in B_{\frac \eps 4}(\eps z_i) \backslash B_{\aeps\rho_z}(\eps z)$, for some $z \in n_\eps(D)$}\\
0  \ \ \ &\text{if $ x \in B_{\eps^3\rho_i}(\eps z_i)$, for some $z_i \in n_\eps(D)$}\\
1 \ \ \ \ \ \ &\text{ otherwise}\\
\end{cases}
\end{align}
For each $z \in n_\eps(D)$, the function $w_{z,\eps}$ is as in \eqref{def.harmonic.annuli}. We remark that each $w_{\eps,z}$ admits the explicit formulation
 \begin{align}\label{explicit.formula}
w_{z,\varepsilon}(x) &= \frac{(\aeps\rho_z)^{-(d-2)}- |x-  \eps z_i|^{-(d-2)}}{(\aeps\rho_z)^{-(d-2)} -(\frac \eps 4)^{-(d-2)}} \ \ \ \text{in $B_{\frac \eps 4}(\eps z) \backslash B_{\aeps\rho_z}(\eps z)$.}
\end{align}

\smallskip

For $k \in \N$, let $\{ Q_{k,z}\}_{z\in N_k}$ be as in the previous subsection. For every $z \in N_k$, we define the random variables
\begin{align}\label{averaged.sum}
S_{k,z}:= \frac{1}{k^d}\sum_{w \in N_{k,z}} Y_{\eps, w} \ \ \ \ \ Y_{\eps,w}:= \rho_w^{d-2} \frac{1}{1- 4^{d-2}\eps^{2} \rho_w^{d-2}}.
\end{align}

\bigskip

\begin{lem}\label{det.estimate}
Let $\delta \in (0, \frac{2}{d-2}]$ be fixed. Then for every $\eps >0$ and $k \in \N$ with $ k\eps \leq 1$ the following inequality holds: If $u_\eps, u$ are as in Theorem \ref{t.main} and $W_\eps, w_\eps$ as in \eqref{corrector} and \eqref{oscillating.function}, respectively,  then
\begin{align*}
\| u_\varepsilon - W_\varepsilon u \|_{H^1_0(D)} &\lesssim \biggl( (k \eps)^2 \eps^d \sum_{z \in \Phi_{\delta}^\eps(D)} \rho_z^{2(d-2)}  +  \eps^d \sum_{z \in \Phi^\eps(D) \backslash n^\eps(D)} \rho_z^{(d-2)}\biggr)^{\frac 1 2}\\
&\quad + \biggl(\avsum_{z \in \mathring{N}_k} (S_{k,z} - \E_\rho \bigl[ \rho^{d-2} \bigr])^2 + (k\eps)^3 \avsum_{z \in N_k \backslash \mathring{N}_{k}} (S_{k,z}  - \E_\rho \bigl[ \rho^{d-2} \bigr])^2  \biggr)^{\frac 1 2}.
\end{align*}
\end{lem}

\bigskip

The next lemma is a simple adaptation of \cite{Kacimi_Murat} to our definition of corrector $W_\eps$ and oscillating test function $w_\eps$:

\begin{lem}\label{det.KM}
Let $\delta \in (0; \frac{2}{d-2})$ be fixed; let $u_\eps, u, w_\eps$ and $W_\eps$ be as in Lemma \ref{det.estimate}. Then
\begin{align*}
\| u_\varepsilon - W_\varepsilon u \|_{H^1_0(D)}^2 \lesssim  \| w_\eps - 1\|_{L^2(D)}^2 + \|\nabla(w_\eps- W_\eps) \|_{L^2(\R^d)}^2 + \| \mu_\eps - C_0 \|_{H^{-1}(D)}^2,
\end{align*}
with
\begin{align}\label{def.mu.eps}
\mu_\eps := \sum_{z \in \Phi_\delta^\eps(D)} \partial_n w_{\eps,z} \, \delta_{\partial B_{\frac \eps 4}(\eps z)}.
\end{align}
\end{lem}

\bigskip

\begin{proof}[Proof of Lemma \ref{det.estimate}]
The statement follows from Lemma \ref{det.KM}, provided that we show that
\begin{align}\label{L.2.norm}
\| \nabla(w_\eps - W_\eps) \|_{L^2(D)}^2 + \| w_{\eps} - 1 \|_{L^2(D)}^2  \lesssim  \eps^{d+2} \sum_{z \in n_\eps(D)} \rho_z^{d-2} + \eps^d \sum_{\Phi^\eps(D) \backslash n^\eps(D)} \rho_z^{d-2}
\end{align}
and that for every $\eps >0$ and $k \in \N$ such that $k \eps \leq 1$ 
\begin{equation} \label{H.minus}
\begin{aligned}
 \| \mu_\eps - C_0 \|_{H^{-1}(D)}^2 &\lesssim (k\eps)^2 \eps^d \sum_{z \in \Phi_{\delta}^\eps(D)} \rho_z^{2(d-2)} \\
 &\quad \quad + \avsum_{z \in \mathring{N_k}} (S_{k,z} - \E_\rho \bigl[ \rho^{d-2} \bigr])^2 + (k\eps)^3 \avsum_{z \in N_k \backslash \mathring{N}_{k}} (S_{k,z}  - \E_\rho \bigl[ \rho^{d-2} \bigr])^2 .
 \end{aligned}
 \end{equation}

\smallskip

We first argue \eqref{L.2.norm}: By definition \eqref{oscillating.bad} for $w_b^\eps$ and Lemma \ref{l.geometry.periodic}, we have that
\begin{align}\label{grad.w.bad}
\| \nabla w_b^\eps \|_{L^2(\R^d)}^2 \lesssim \eps^d \sum_{\Phi^\eps(D) \backslash n^\eps(D)} \rho_z^{d-2}.
\end{align}
Since by Lemma \ref{l.geometry.periodic}  the sets $\bigcup_{z \in n^\eps(D)} B_{\frac \eps 4}(\eps z)$ and $D_b^\eps$ are disjoint, we appeal to \eqref{oscillating.function} to estimate
\begin{align}\label{L.2.norm.1}
\| w_\eps - 1 \|_{L^2(D)}^2 = \sum_{z_i \in n_\eps(D)} \| w_\eps^g - 1 \|_{L^2(B_{\frac \eps 4}(\eps z_i))}^2 + \| w_\eps^b -1 \|_{L^2(D_\eps^b \cap D)}^2.
\end{align}
The function $w_\eps^g-1$ vanishes on $\bigcup_{z \in n_\eps(D)} \partial B_{\frac \eps 4}(\eps z )$: Since the balls $\{ B_{\frac{\eps}{4}}(\eps z)\}_{z \in n^\eps(D)}$ are all disjoint, Poincar\'e's inequality in each ball $B_{\frac \eps 4}(\eps z)$ yields
\begin{align*}
\| w_\eps^g - 1 \|_{L^2(D)}^2 \lesssim \eps^2 \sum_{z \in n_\eps(D)} \| \nabla w_\eps^g \|_{L^2(B_{\frac \eps 4}(\eps z))}^2.
\end{align*}
Using definition \eqref{oscillating.good}, we may rewrite
\begin{align*}
\| w_\eps^g - 1 \|_{L^2(D)}^2 \lesssim \eps^{d+2}  \sum_{z \in n_\eps(D)} \rho_z^{d-2},
\end{align*}
and, inserting this into \eqref{L.2.norm.1}, also
\begin{align}\label{L.2.norm.2}
\| w_\eps - 1 \|_{L^2(D)}^2 = \eps^{d+2}  \sum_{z \in n_\eps(D)} \rho_z^{d-2} + \| w_\eps^b -1 \|_{L^2(D_\eps^b \cap D)}^2.
\end{align}
To conclude the proof of \eqref{L.2.norm} for $w_\eps -1$, it thus remains to estimate the last term on the right-hand side. By construction  (c.f. \eqref{oscillating.bad}), it holds $w_\eps^b -1 = 0 $ on $\partial D_\eps^b$; appealing to Lemma \ref{l.geometry.periodic}, we also have that $D^\eps_b \subset \{ x \in \R^3 \, \colon \, \mathop{dist}(x, D) \leq 2 \}$. We thus apply Poincar\'e's inequality in this set and conclude that
\begin{align*}
\| w_\eps^b - 1 \|_{L^2(D_\eps^b \cap D)}^2 \lesssim  \|\nabla w_\eps^b \|_{L^2(D_\eps^b)}^2 \stackrel{\eqref{grad.w.bad}}{\lesssim}  \eps^d \sum_{\Phi^\eps(D) \backslash n^\eps(D)} \rho_z^{d-2}.
\end{align*}
To establish  \eqref{L.2.norm} for $w_\eps -1$, it only remains to combine this last inequality with \eqref{L.2.norm.2}.

\smallskip

We now argue \eqref{L.2.norm} for $\nabla(w_\eps - W_\eps)$:  By definition \eqref{thinning.psi} and \eqref{good.set.periodic} of Lemma \ref{l.geometry.periodic}, it holds
\begin{align}\label{inclusion.set}
n^\eps(D) \subset \Phi^\eps_\delta(D).
\end{align}
Thanks to definition \eqref{oscillating.function} for $w_\eps$ and the fact that, by Lemma \ref{l.geometry.periodic} the support of $\nabla w_g^\eps$ and $\nabla w_b^\eps$ is disjoint, we use the triangle inequality to infer that
\begin{align}\label{comparison.periodic}
\| \nabla ( w_\eps - W_\eps ) \|_{L^2(D)}^2 &\lesssim \| \nabla ( w_g^\eps- W_\eps ) \|_{L^2(D)}^2 + \| \nabla w_b^\eps\|_{L^2(D)}^2\\
& \stackrel{\eqref{grad.w.bad}}{\lesssim} \| \nabla ( w_g^\eps- W_\eps ) \|_{L^2(D)}^2 + \eps^d \sum_{z \in \Phi^\eps(D) \backslash n^\eps(D)} \rho_z^{d-2}.
\end{align}
Comparing definition \eqref{oscillating.good} for $w_g^\eps$ with definition \eqref{corrector} for $W_\eps$ and using inclusion \eqref{inclusion.set}, we observe that
\begin{align*}
\nabla (w_g^\eps - W_\eps) = \sum_{\Phi^\eps_\delta(D) \backslash n^\eps(D)} \nabla W_\eps \1_{B_{\frac \eps 4}(\eps z)}.
\end{align*}
Since the balls $\{ B_{\frac \eps 4}(\eps z) \}_{z \in  \Phi^\eps_\delta(D)}$ are disjoint, the previous identity and the triangle inequality imply that
\begin{align*}
\|\nabla (w_g^\eps - W_\eps)\|_{L^2(D)}^2& \lesssim  \sum_{\Phi^\eps_\delta(D) \backslash n^\eps(D)} \|\nabla w_{\eps,z} \|_{L^2(B_{\frac \eps 4}(\eps z))}^2\\
&\stackrel{\eqref{def.harmonic.annuli}}{=}  \sum_{\Phi^\eps_\delta(D) \backslash n^\eps(D)} \capacity(B_{\aeps\rho_z}(\eps z) ; B_{\frac \eps 4}(\eps z)) \stackrel{\eqref{thinning.psi}}{\lesssim}  \eps^d \sum_{\Phi^\eps(D) \backslash n^\eps(D)}\rho_z^{d-2}.
\end{align*}
Inserting this bound into \eqref{comparison.periodic} yields \eqref{L.2.norm} also for the norm of $\nabla(w_\eps- W_\eps)$.

\bigskip

We now turn to \eqref{H.minus} and claim that we may apply Lemma \ref{Kohn_Vogelius} with $M= \mu_\eps$, $\mathcal{Z}= \{ \eps w \}_{w \in \Phi^\eps_\delta(D)}$, $\mathcal{X}=\{ \aeps \rho_w\}_{w \in \Phi^\eps_\delta(D)}$ and $r_w\equiv \frac \eps 4$ for every $w \in \Phi^\eps_\delta(D)$. We use as covering $\{K_j\}_{j \in J}$ the sets $\{ Q_{k, z}\}_{z \in N_k}$. Conditions \eqref{well.defined.cells}  and \eqref{contains.balls} are satisfied thanks to \eqref{thinning.psi} and by construction (see Subsection \ref{sub.covering}), respectively. Appealing to Lemma \ref{Kohn_Vogelius}, we therefore have that
\begin{align*}
\|  \mu_\eps - m_k \|_{H^{-1}(D)}^2 \lesssim  (k\eps)^2 \eps^d \sum_{z \in \Phi_{\delta}^\eps(D)} \rho_z^{d-2}, \ \ \ m_k \stackrel{\eqref{averaged.sum}}{ =} c_d \sum_{z \in N_{k}} S_{k,z}  \1_{Q_{k, z}}.
\end{align*}

By the triangle inequality and the previous estimate, we thus bound
\begin{align}\label{H.minus.1}
\| \mu_\eps - C_0 \|_{H^{-1}(D)}^2 \leq (\eps k)^2 \eps^d \sum_{z \in \Phi_{\delta}^\eps(D)} \rho_z^{2(d-2)} +  \| m_k - C_0 \|_{H^{-1}(D)}^2
\end{align}
so that, to prove \eqref{H.minus}, it only remains to control the last term on the right-hand side above. We do this by observing that for each $\phi \in H^1_0(D)$ we have
\begin{align}
|\langle m_k - m ; \phi \rangle | \simeq  | \sum_{z \in N_k} (S_{k,z}  -\E_\rho \bigl[ \rho^{d-2} \bigr]) \int_{Q_{k,z} \cap D}\phi |
\end{align}
and, by the triangle inequality, also
\begin{equation}
\begin{aligned}\label{H.minus.5}
|\langle m_k - m ; \phi \rangle | &\stackrel{\eqref{interior.cubes}}{\lesssim} | \sum_{z \in \mathring{N}_{k}} (S_{k,z}  -\E_\rho \bigl[ \rho^{d-2} \bigr]) \int_{Q_{k,z}}\phi |\\
&\quad\quad  + | \sum_{z \in N_k \backslash \mathring{N}_k} (S_{k,z}  -\E_\rho \bigl[ \rho^{d-2} \bigr]) \int_{Q_{k,z} \cap D}\phi |.
\end{aligned} 
\end{equation}

\smallskip

We claim that 
\begin{align}\label{H.minus.6.b}
| \sum_{z \in \mathring{N}_{k}} (S_{k,z}  -\E_\rho \bigl[ \rho^{d-2} \bigr]) \int_{Q_{k,z}}\phi |  \lesssim  \bigl( \avsum_{z \in \mathring{N}_k} (S_{k,z}  - \E_\rho \bigl[ \rho^{d-2} \bigr])^2 \bigr)^{\frac 1 2} \bigl( \int_D |\nabla \phi |^2 \bigr)^{\frac 12}. 
\end{align}
This is an easy consequence of the properties of the covering $\{ Q_{k,z} \}_{z \in N_k}$ of $D$, \eqref{number.N.k}, together with Cauchy-Schwarz's inequality and Poincar\'e's inequality for $\phi$ in $D$.

\smallskip

We now turn to the second term in \eqref{H.minus.5}. We note that, by definition \eqref{interior.cubes}, the set 
$$
\bigcup_{z \in N_k \backslash \mathring{N}_{k}} Q_{k, z} \subset \{ x \in \R^d \, \colon \, \mathop{dist}( x; \partial D) \leq 4k \eps \}.
$$
Since $\phi \in H^1_0(D)$ and $D$ is a smooth and bounded set, we may appeal to Poincar'e's inequality {\red \cite{???}}  to bound
\begin{align}\label{poincare.boundary}
\bigl(\sum_{z \in N_k \backslash \mathring{N}_k} \int_{Q_{k,z}}|\phi|^2  \bigr)^{\frac 1 2} \lesssim (k \eps) \bigl(\int_D |\nabla \phi|^2  \bigr)^{\frac 1 2}.
\end{align}
Appealing once again Cauchy-Schwarz's inequality and using the above estimate, we control
\begin{align*}
| \sum_{z \in N_k \backslash \mathring{N}_{k}} (S_{k,z}  -\E_\rho \bigl[ \rho^{d-2} \bigr]) \int_{K_{\alpha,z} \cap D}\phi | \lesssim  \bigl((k\eps)^{d+2} \sum_{z \in N_k \backslash \mathring{N}_{k}} (S_{k,z}  - \E_\rho \bigl[ \rho^{d-2} \bigr])^2 \bigr)^{\frac 1 2} \bigl(\int_D |\nabla \phi|^2 \bigr)^{\frac 12}.
\end{align*}
Hence, provided $k \eps \leq 1$, we may appeal to  \eqref{number.boundary.cubes} and infer that
\begin{align*}
| \sum_{z \in N_k \backslash \mathring{N}_{k}} (S_{k,z}  -\E_\rho \bigl[ \rho^{d-2} \bigr]) \int_{Q_{\alpha,z} \cap D}\phi |
&{\lesssim}  \bigl((k\eps)^3 \avsum_{z \in N_k \backslash \mathring{N}_{k}} (S_{k,z}  - \E_\rho \bigl[ \rho^{d-2} \bigr])^2 \bigr)^{\frac 1 2} \bigl(\int_D |\nabla \phi|^2 \bigr)^{\frac 12}.
\end{align*}
Combining this with \eqref{H.minus.6.b} and \eqref{H.minus.5} allows us to infer that for every $\phi \in H^1_0(D)$
\begin{align*}
|\langle m_k - m ; \phi \rangle | {\lesssim}  \bigl((k\eps)^3 \avsum_{z \in N_k \backslash \mathring{N}_{k}} (S_{k, z}  - \E_\rho \bigl[ \rho^{d-2} \bigr])^2  +  \avsum_{z \in \mathring{N}_k} (S_{k,z}  - \E_\rho \bigl[ \rho^{d-2} \bigr])^2 \bigr)^{\frac 1 2} \bigl( \int_D |\nabla \phi |^2 \bigr)^{\frac 12},
\end{align*}
or, equivalently, that
\begin{align}\label{H.minus.7}
\| m_k - m \|_{H^{-1}(D)}^2 {\lesssim} (k\eps)^3 \avsum_{z \in N_k \backslash \mathring{N}_{k}} (S_{k,z}  - \E_\rho \bigl[ \rho^{d-2} \bigr])^2  +  \avsum_{z \in \mathring{N}_k} (S_{k,z}  - \E_\rho \bigl[ \rho^{d-2} \bigr])^2.
\end{align}
This, together with \eqref{H.minus.1}, establishes \eqref{H.minus}. The proof of Lemma \ref{det.estimate} is complete.
\end{proof}

\bigskip

\begin{proof}[Proof of Lemma \ref{det.KM}] The argument for this lemma is very similar to the one of \cite[Theorem 3.1]{Kacimi_Murat}. Since $f \in L^\infty$ and $D$ is smooth, by standard elliptic regularity we infer that $u \in W^{2,\infty}(D)$. By computing the (distributional) Laplacian of $u_\eps - w_\eps u$ we obtain that in $D^\eps$
\begin{align}\label{distributional.laplacian}
-\Delta ( u_\eps - &w_\eps u) = (C_0 + \Delta w_\eps) u - 2 \nabla \cdot ( (1-w_\eps) \nabla u ) + (1-w_\eps)\Delta u 
\end{align}
We now smuggle the term $(-\Delta W_\eps) u \in H^{-1}(D)$ in the right-hand side so that the previous identity turns into
\begin{align*}
-\Delta &( u_\eps - w_\eps u) \\
&=  (C_0 + \Delta W_\eps) u - \Delta( W_\eps-  w_\eps) u - 2 \nabla \cdot ( (1-w_\eps) \nabla u ) + (1-w_\eps)\Delta u  \ \ \ \ \text{in $D_\eps$.}
\end{align*}
We stress that, since $u \in W^{2,+\infty}(D)\cap H^1_0(D)$, $u_\eps \in H^1_0(D^\eps)$ ,$w_\eps \in H^1(D)$, the above equation holds in the sense that for every $\phi \in H^1_0(D_\eps)$ 
\begin{align}\label{weak.formulation}
\int \nabla \phi \cdot \nabla( u_\eps - w_\eps u) &= \langle C_0 + \Delta W_\eps ; u \phi \rangle + \int \nabla(W_\eps- w_\eps) \cdot \nabla(u \phi) \\
&\quad\quad + 2\int (1-w_\eps) \nabla u \cdot \nabla \phi + \int (1-w_\eps) \Delta u \, \phi  .
\end{align}
Since the balls $\{ B_{\frac \eps 4}(\eps z)\}_{z \in \Phi_\delta^\eps(D)}$ are all mutually disjoint, by definition \eqref{corrector} and equations \eqref{def.harmonic.annuli} we have that
$$
-\Delta W_\eps:= \sum_{z \in \Phi_\delta^\eps(D)} \partial_n w_{\eps,z} (\1_{\partial B_{\frac \eps 4}(\eps z)} - \1_{\partial B_{\aeps\rho_z}(\eps z)}).
$$
Since $\phi \in H^1_0(D^\eps)$ and therefore it vanishes on the spheres $\{\partial B_{\aeps\rho_z}(\eps z)\}_{z \in \Phi_\delta^\eps(D)}$, the above identity implies that
$$
 \langle \Delta W_\eps  ; u\phi \rangle  = - \sum_{z \in \Phi_\delta^\eps(D)} \int_{\partial B_{\frac \eps 4}(\eps z)} \partial_n w_{\eps,z} u \phi \stackrel{\eqref{def.mu.eps}}{=}- \langle \mu_\eps ; u \phi \rangle.
$$
Inserting this last identity in \eqref{weak.formulation}, we infer that 
\begin{align*}
\int \nabla \phi \cdot \nabla( u_\eps - w_\eps u) &= \langle C_0 - \mu_\eps; u \phi \rangle + \int \nabla(W_\eps- w_\eps) \cdot \nabla(u \phi) \\
&\quad\quad + 2\int (1-w_\eps) \nabla u \cdot \nabla \phi + \int (1-w_\eps) \Delta u \, \phi .
\end{align*}
We now choose $\phi = u_\eps- w_\eps u$ and apply H\"older's and Poincar\'e's inequalities to bound
\begin{align*}
\| u_\varepsilon - w_\varepsilon u \|_{H^1_0(D)}^2 \lesssim \| u \|_{W^{2,\infty}}^2\bigl( \| w_\eps - 1\|_{L^2(D)}^2 + \| \nabla (w_\eps- W_\eps) \|_{L^2(D)}^2 + \| \mu_\eps - \mu \|_{H^{-1}(D)}^2 \bigr).
\end{align*}
To obtain the claim of Lemma \ref{det.KM} it remains to use that, by the triangle inequality and H\"older's inequality, we have
\begin{align*}
\|\nabla( u_\eps - W_\eps u) \|_{L^2(D)} \leq \|u \|_{W^{2,\infty}} \| W_\eps - w_\eps \|_{H^1(\R^d)} + \| u_\varepsilon - w_\varepsilon u \|_{H^1_0(D)}
\end{align*}
and that, by definitions \eqref{corrector} and \eqref{oscillating.function}, the difference $W_\eps- w_\eps$ is compactly supported in $\{ x \in \R^d \, \colon \, \mathop{dist}(x; \partial D) \leq 4 \}$ (see also Lemma \ref{l.geometry.periodic}).
\end{proof}

\bigskip

\subsection{Annealed estimates (Proof of Theorem \ref{t.main}, $(a)$)}
In this subsection we rely on the quenched estimate of Lemma \ref{det.estimate} to prove the  statement of Theorem \ref{t.main} in the case of periodic centres. The first ingredient is the following annealed bound:

\bigskip

\begin{lem}\label{l.capacity.average} Let $(\Phi , \mathcal{R})$ satisfy the assumptions of Theorem \ref{t.main}, $(a)$. For every $\delta \in (0, \frac{2}{d-2}]$, let $n^\eps(D) \subset \Phi^\eps(D)$ be the random subset constructed in Lemma \ref{l.geometry.periodic}.  Then
\begin{align*}
\E \bigl[  \eps^d \sum_{z \in \Phi^\eps(D) \backslash n^\eps(D)} \rho_z^{d-2} \bigr] \lesssim  \eps^{(\frac{2}{d-2} -\delta)\beta}. 
\end{align*}
\end{lem}

\bigskip

\begin{proof}[Proof of Theorem \ref{t.main}, $(a)$]
By the assumptions on $D$, we may find a constant $c=c(D) \leq 1$ such that for $\eps > 0$, and $k \in \N$ such that $\eps k \leq c$, the cube $Q_{k,0} \subseteq D$. We restrict to the values of $k\in\N$ satisfying the previous bound.

\smallskip

Combining Lemma \ref{det.estimate} and Lemma \ref{l.capacity.average}, we bound for every $\eps> 0$ and $k\in \N$ as above
\begin{align*}
\E \bigl[ \| u_\varepsilon - W_\varepsilon u \|_{H^1_0(D)}^2 \bigr] &\lesssim  (k\eps)^2  \E \bigl[ \eps^d \sum_{z \in \Phi_{\delta}^\eps(D)} \rho_z^{2(d-2)} \bigr] + \E  \bigl[ \avsum_{z \in \mathring{N}_k} (S_{k,z} -  \E \bigl[ \rho^{d-2} \bigr])^2 \bigr] \\
&\quad\quad + (\eps k)^3 \E  \bigl[ \avsum_{z \in N_k \backslash \mathring{N}_k} (S_{k,z} - \E \bigl[ \rho^{d-2} \bigr])^2 \bigr] +\eps^{(\frac{2}{d-2} -\delta)\beta}.
\end{align*}
Since the sets $N_k, \mathring{N}_k$ are deterministic, and $\{ S_{k,z}\}_{z \in \mathring{N}_k}$ are identically distributed, we infer that 
\begin{align*}
\E \bigl[ \| u_\varepsilon - W_\varepsilon u \|_{H^1_0(D)}^2 \bigr] &\lesssim  (k\eps)^2  \E \bigl[ \eps^d \sum_{z \in \Phi_{\delta}^\eps(D)} \rho_z^{2(d-2)} \bigr] + \E  \bigl[(S_{k,0} - \E \bigl[ \rho^{d-2} \bigr])^2 \bigr] \\
&\quad\quad + (\eps k)^3 \avsum_{z \in N_k \backslash \mathring{N}_k}  \E  \bigl[ (S_{k,z} - \E \bigl[ \rho^{d-2} \bigr])^2 \bigr] + \eps^{(\frac{2}{d-2} -\delta)\beta},
\end{align*}
We observe that 
$$
 \E \bigl[ \eps^d \sum_{z \in \Phi_{\delta}^\eps(D)} \rho_z^{2(d-2)} \bigr]  \lesssim \E \bigl[ \rho^{2(d-2)} \1_{\rho < \eps^{-\frac{2}{d-2}+\delta}} \bigr].
$$
This identity is, indeed, an easy consequence of the definition \eqref{thinning.psi} and of the fact that $\Phi^\eps(D)$ is deterministic with $\#\Phi^\eps(D) \lesssim \eps^{-d}$. The previous two displays thus imply that
\begin{align}\label{main.1}
\E \bigl[ \| u_\varepsilon - W_\varepsilon u \|_{H^1_0(D)}^2 \bigr] &\lesssim  (k\eps)^2  \E \bigl[ \rho^{2(d-2)} \1_{\rho < \eps^{-\frac{2}{d-2}+\delta}} \bigr] + \E  \bigl[(S_{k,0} - \E \bigl[ \rho^{d-2} \bigr])^2 \bigr] \\
&\quad\quad + (\eps k)^3 \avsum_{z \in N_k \backslash \mathring{N}_k}  \E  \bigl[ (S_{k,z} - \E \bigl[ \rho^{d-2} \bigr])^2 \bigr] +\eps^{(\frac{2}{d-2} -\delta)\beta}.
\end{align}

\smallskip

We now claim that if we choose  $k \leq \eps^{-\frac{2}{d+2}}$, then
\begin{equation}
\begin{aligned}\label{main.0.b}
\E \bigl[ \| u_\varepsilon - W_\varepsilon u \|_{H^1_0(D)}^2 \bigr] &\lesssim  (k\eps)^2\E \bigl[ \rho^{2(d-2)} \1_{\rho < \eps^{-\frac{2}{d-2}+\delta}} \bigr] \\
& \quad\quad\quad + k^{-d} \mathop{Var}(Y_{\eps,0} \1_{\rho< \eps^{-\frac{2}{d-2} +\delta}} )  +\eps^{(\frac{2}{d-2} -\delta)\beta},
\end{aligned}
\end{equation}
where $Y_{\eps,0}$ is defined as in \eqref{averaged.sum}. We remark that for $k$ as above, we have that $\eps k \to 0$ when $\eps \downarrow 0$ and therefore that $\eps k \leq c$ for  $\eps$ is small enough (only depending on $D$ and $d$).  We begin by showing how to conclude the proof of the theorem provided the estimate in the previous display holds. 

\smallskip

Let us first assume that \eqref{integrability.radii} holds with $\beta \geq d-2$; in this case, we have that 
$$
\E\bigl[  \rho^{2(d-2)}  \bigr] + \E\bigl[ Y_{\eps,0}^{2(d-2)} \1_{\rho < \eps^{-\frac{2}{d-2}+\delta}} \bigr] \lesssim 1
$$ 
and therefore that
\begin{align}\label{important.estimate}
\E \bigl[ \| u_\varepsilon - W_\varepsilon u \|_{H^1_0(D)}^2 \bigr] &\lesssim  (k\eps)^2 + k^{-d} +  \eps^{(\frac{2}{d-2}-\delta) \beta}.
\end{align}
Estimate of Theorem \ref{t.main} for $\beta \geq d-2$ follows from this inequality if we minimize the right-hand side above in $k$, i.e. if we choose $k= \lfloor \eps^{-\frac{2}{d+2}} \rfloor$, and set $\delta$ as in Theorem \ref{t.main}.

\smallskip

Let us now assume that $\beta < d-2$ in \eqref{integrability.radii}: In this case, we bound
\begin{align}
\mathop{Var}(Y_{\eps,0} \1_{\rho< \eps^{-\frac{2}{d-2}+\delta}} ) + \E\bigl[ \rho^{2(d-2)} \1_{\rho < \eps^{-\frac{2}{d-2}+\delta}} \bigr] \lesssim \eps^{-(\frac{2}{d-2}- \delta)(d-2-\beta)}
\end{align}
so that \eqref{main.0.b} turns into
\begin{align}
\E \bigl[ \| u_\varepsilon - W_\varepsilon u \|_{H^1_0(D)}^2 \bigr] &\lesssim \bigl( (k\eps)^2 + k^{-d} \bigr) \eps^{-(\frac{2}{d-2}- \delta)(d-2-\beta)} +  \eps^{(\frac{2}{d-2}-\delta) \beta}.
\end{align}
Also in this case, we infer the estimate of Theorem \ref{t.main} by minimizing the right-hand side in $k$ and $\delta$, i.e. choosing $k= \lfloor \eps^{-\frac{2}{d+2}} \rfloor$ and $\delta$ as in Theorem \ref{t.main}. 

\bigskip

To complete the proof of the theorem it only remains to argue \eqref{main.0.b} from \eqref{main.1}. We first tackle the second term on the right-hand side of \eqref{main.1} and show that
\begin{align}\label{main.3}
\E  \bigl[(S_{k,0} - \E \bigl[ \rho^{d-2} \bigr])^2 \bigr] &\lesssim  k^{-d} \mathop{Var}(Y_{w,\eps}  \1_{\rho < \eps^{-\frac{2}{d-2}+\delta}})+ \eps^4 \E \bigl[ \rho^{2(d-2)}   \1_{\rho < \eps^{-\frac{2}{d-2} +\delta}} \bigr]^2 \\
&\quad \quad + \eps^{2(\frac{2}{d-2} -\delta)\beta}.
\end{align}
We begin by remarking that the definition of $\Phi^\eps_{\delta}(D)$ and of $N_{k,z}$ (see \eqref{thinning.psi} and \eqref{number.covering}), implies that
\begin{align}\label{N.k.lattice}
N_{k,z} = \bigl\{ w \in \Z^d \, \colon \, \eps w  \in Q_{k,z} \cap D, \ \ \aeps\rho_w < \eps^{1+\delta} \bigr\}.
\end{align}
Since $0 \in \mathring{N}_k$, this last identity allows us to rewrite
\begin{align*}
S_{k,0} -  \E_\rho \bigl[\rho^{d-2} \bigr] &\stackrel{\eqref{averaged.sum}}{=} \avsum_{w \in \Z^d \atop \eps w \in Q_{k, 0}}Y_{w,\eps} \1_{\rho_w < \eps^{-\frac{2}{d-2}+\delta}} - \E_\rho \bigl[\rho^{d-2} \bigr] \\
&=  \avsum_{w \in \Z^d \atop \eps w \in Q_{k, 0}} (Y_{w,\eps}\1_{\rho_w < \eps^{-\frac{2}{d-2}+\delta}} - \E_\rho \bigl[\rho^{d-2} \1_{\rho < \eps^{-\frac{2}{d-2}+\delta}} \bigr] \bigr) + \E_\rho \bigl[\rho^{d-2} \1_{\rho \geq \eps^{-\frac{2}{d-2}+\delta}} \bigr] .
\end{align*}
Hence,
\begin{align}\label{main.3.c}
\E \bigl[ (S_{k,0} - \E_\rho \bigl[\rho^{d-2} \bigr] )^2\bigr] &\lesssim \E \bigl[ (\avsum_{w \in \Z^d \atop \eps w \in Q_{k, 0}}Y_{w,\eps} \1_{\rho_w < \eps^{-\frac{2}{d-2} +\delta}} - \E_\rho \bigl[\rho^{d-2}\1_{\rho < \eps^{-\frac{2}{d-2}+\delta}}\bigr] )^2 \bigr] \\
&\quad \quad + \E_\rho \bigl[\rho^{d-2} \1_{\rho \geq \eps^{-\frac{2}{d-2}+\delta}} \bigr]^2
\end{align}
Since Chebyshev's inequality and assumption \eqref{integrability.radii} we have
\begin{align}\label{Cheb}
\E_\rho \bigl[\rho^{d-2} \1_{\rho \geq \eps^{-\frac{2}{d-2}+\delta}} \bigr]\lesssim \eps^{(\frac{2}{d-2} -\delta)\beta},
\end{align}
we rewrite \eqref{main.3.c} as
\begin{align}\label{main.3.c}
\E \bigl[ (S_{k,0} - \E_\rho \bigl[\rho^{d-2} \bigr] )^2\bigr] &\lesssim \E \bigl[ (\avsum_{w \in \Z^d \atop \eps w \in Q_{k, 0}}Y_{w,\eps} \1_{\rho_w < \eps^{-\frac{2}{d-2} +\delta}} - \E_\rho \bigl[\rho^{d-2}\1_{\rho < \eps^{-\frac{2}{d-2}+\delta}}\bigr] )^2 \bigr] \\
&\quad\quad +  \eps^{2(\frac{2}{d-2} -\delta)\beta}.
\end{align}
Using the independence of the random variables $\{\rho_z \}_{z \in \Phi}$, and the fact that for $\Phi =\Z^d$ we have $N^\eps(Q_{k,0}) = k^{-d}$ (c.f. \eqref{notation.psi}), we decompose \begin{equation}
\begin{aligned}\label{main.3.b}
\E \bigl[ (\avsum_{w \in \Z^d \atop \eps w \in Q_{k, 0}}Y_{w,\eps} \1_{\rho_w < \eps^{-\frac{2}{d-2} +\delta}}& - \E_\rho \bigl[\rho^{d-2}\1_{\rho < \eps^{-\frac{2}{d-2}+\delta}}\bigr] )^2 \bigr] \lesssim  k^{-d} \mathop{Var}(Y_{0,\eps}  \1_{\rho < \eps^{-\frac{2}{d-2} +\delta}}) \\
& + k^{-2d}\sum_{w \in \Z^d \atop \eps w \in Q_{k, 0}}\sum_{\tilde w \in \Z^d \backslash\{ w \} \atop \eps \tilde w \in Q_{k, 0}} (\E \bigl[  Y_{\eps,0}- \E_\rho\bigl[ \rho^{d-2}\1_{\rho < \eps^{-\frac{2}{d-2}+\delta}}\bigr]\bigr])^2\end{aligned}
\end{equation}
We now observe that, by \eqref{averaged.sum} and the triangle inequality, it holds
\begin{align}
|\E \bigl[  Y_{\eps,0}- \E_\rho\bigl[ \rho^{d-2}\1_{\rho < \eps^{-\frac{2}{d-2}+\delta}}\bigr]| \lesssim \eps^2 \E \bigl[ \rho^{2(d-2)}   \1_{\rho < \eps^{-\frac{2}{d-2} +\delta}} \bigr].
\end{align}
To obtain \eqref{main.3} it only remains to insert the inequality above into \eqref{main.3.b}.

\bigskip

We now turn to the remaining term in \eqref{main.1} and argue that
\begin{align}\label{main.4}
(\eps k)^3 \avsum_{z \in N_k \backslash \mathring{N}_k}  \E  \bigl[ (S_{k,z} - \E \bigl[ \rho^{(d-2)} \bigr])^2 \bigr] \lesssim (k\eps)^2\E \bigl[ \rho^{2(d-2)} \1_{\rho < \eps^{-\frac{2}{d-2} +\delta}} \bigr].
\end{align}
By the triangle inequality and assumption \eqref{integrability.radii}, the left-hand side is bounded by
\begin{align}\label{main.4.b}
(\eps k)^3 \avsum_{z \in N_k \backslash \mathring{N}_k}  \E  \bigl[ (S_{k,z} - \E \bigl[ \rho \bigr])^2 \bigr] \lesssim (\eps k)^3 + (\eps k)^3 \avsum_{z \in N_k \backslash \mathring{N}_k}  \E  \bigl[ S_{k,z}^2 \bigr]
\end{align}
To establish \eqref{main.4} from this it suffices to remark that, by \eqref{averaged.sum} and \eqref{N.k.lattice}, we have 
$$
|S_{k,z}|^2 \lesssim \avsum_{w \in \Z^3 \atop \eps w \in Q_{k,z}\cap D}  \rho_w^{2(d-2)}\1_{\rho_w < \eps^{-\frac{2}{d-2} +\delta}}
$$
so that this, and the fact that the random variables $\{ \rho_z \}_{z \in \Phi^\eps(D)}$ are identically distributed, yields
\begin{align*}
 \avsum_{z \in N_k \backslash \mathring{N}_k}  \E  \bigl[ (S_{k,z})^2 \bigr]\lesssim \E  \bigl[\1_{\rho < \eps^{-\frac{2}{d-2} +\delta}} \rho^{2(d-2)} \bigr].
\end{align*}
Inserting this into \eqref{main.4.b} implies \eqref{main.4}. 

To establish \eqref{main.0.b}  it remains to combine \eqref{main.4}, \eqref{main.3} and \eqref{main.1} and use that, for $k \leq \eps^{-\frac{2}{d-2}}$, it holds
$$
\eps^4 \E \bigl[ \rho^{2(d-2)}   \1_{\rho < \eps^{-\frac{2}{d-2} +\delta}} \bigr]^2 \lesssim (\eps k)^2 \E \bigl[ \rho^{2(d-2)}   \1_{\rho < \eps^{-\frac{2}{d-2} +\delta}} \bigr].
$$
 The proof of Theorem \ref{t.main}, $(a)$ is complete.
\end{proof}

\bigskip

\begin{proof}[Proof of Lemma \ref{l.capacity.average}]
We resort to the construction of the set $n^\eps(D)$ implemented in Lemma \ref{l.geometry.periodic}: By \eqref{definition.bad.index}, \eqref{index.set.J} and \eqref{index.set.I.tilde} in the proof of Lemma \ref{l.geometry.periodic} we decompose
\begin{align}\label{capacity.bad.periodic.1}
 \eps^d \sum_{\Phi^\eps(D) \backslash n^\eps(D)} \rho_w^{d-2} = \eps^d \sum_{z \in J^\eps_b} \rho_z^{d-2} +  \eps^d \sum_{z \in \tilde I^\eps_b} \rho_z^{d-2} .
\end{align}
and prove the statement of Lemma \ref{l.capacity.average} for each one of the two sums. We begin with the first one: Using \eqref{index.set.J} we write
\begin{align*}
 \eps^d \sum_{z \in J^\eps_b} \rho_z^{d-2} =  \eps^d \sum_{z \in \Phi^\eps(D)} \rho_z^{d-2} \1_{\rho_z \geq \eps^{-\frac{2}{d-2} +\delta}}.
\end{align*}
Taking the expectation and using that $\{ \rho_z\}_{\Phi^\eps(D)}$ are identically distributed and that $N^\eps(D) \lesssim \eps^{-d}$, we immediately bound
\begin{align}\label{capacity.bad.periodic.sum1}
 \eps^d \E \bigl[ \sum_{z \in J^\eps_b} \rho_z^{d-2} \bigr] \lesssim \E_\rho \bigl[ \rho^{d-2} \1_{\rho \geq \eps^{-\frac{2}{d-2} +\delta}} \bigr] \stackrel{\eqref{Cheb}}{\lesssim} \eps^{(\frac{2}{d-2} - \delta)\beta},
\end{align}
i.e.  the claim of Lemma \ref{l.capacity.average} for the first sum in \eqref{capacity.bad.periodic.1}.

\smallskip

We now turn to the second sum in \eqref{capacity.bad.periodic.1}: By definition \eqref{index.set.I.tilde}, if $z \in \tilde I_b^\eps$, then $\rho_z \leq \eps^{-\frac{2}{d-2} +\delta}$ and there exists an element $w \in J_b^\eps$ such that $\eps |z -w| < \eps + (\aeps\rho_w \wedge 1)$. This allows us to bound
\begin{align}\label{capacity.bad.3}
\eps^d \sum_{z \in \tilde I^\eps_b} \rho_z^{d-2} &\leq \eps^d \sum_{w \in  J^\eps_b}  \sum_{z \in \Phi^\eps(D) \backslash \{ w\}, \atop \eps |z-w| < \eps + \aeps\rho_w \wedge 1} \rho_z^{d-2} \1_{\rho_z < \frac 1 2 \eps^{-\frac{2}{d-2} +\delta}}\\
&=  \eps^d\sum_{w \in \Phi^\eps(D)}  \1_{\rho_w \geq \eps^{-\frac{2}{d-2} +\delta}} \sum_{z \in \Phi^\eps(D) \backslash \{ w\}, \atop \eps |z-w| < \eps+ \aeps\rho_w} \rho_z^{d-2} \1_{\rho_z <  \eps^{-\frac{2}{d-2} +\delta}} .
\end{align}
We now take the expectation and use that $\Phi= \Z^d$ and that $\{\rho_z \}_{z\in \Phi}$ are independent and identically distributed: This implies that
\begin{align}\label{capacity.bad.4}
\E \bigl[ \eps^d \sum_{z \in \tilde I^\eps_b} \rho_z^{d-2} \bigr]& \lesssim \E\bigl[ \eps^d \sum_{w \in \Phi^\eps(D)}  \1_{\rho_w \geq \eps^{-\frac{2}{d-2} +\delta}}  \# \{ z \in \Phi^\eps(D) \backslash \{ w\} \, \colon \, \eps|z-w| < \eps + \aeps\rho_w \wedge 1\} \bigr].
\end{align}
Since for every $w \in J^\eps_b$, the set
$$
\# \{ z \in \Phi^\eps(D) \backslash \{ w\} \, \colon \, \eps|z-w| < \eps + \aeps\rho_w \wedge 1\}  \lesssim 1 + \rho_w^{d-2},
$$
we obtain that
\begin{align}\label{term.I.tilde.periodic}
\E \bigl[ \eps^d \sum_{z \in \tilde I^\eps_b} \rho_z^{d-2} \bigr]& \lesssim   \E\bigl[\eps^d\sum_{w \in \Phi^\eps(D)} \rho_w^{d-2} \1_{\rho_w \geq \eps^{-\frac{2}{d-2} +\delta}}\bigr] \stackrel{\eqref{index.set.J}}{=} \E\bigl[\eps^d \sum_{w \in J^\eps_b}\rho_w^{d-2}\bigr]\\
&\stackrel{\eqref{capacity.bad.periodic.sum1}}{\lesssim} \eps^{(\frac{2}{d-2}-\delta)\beta} .
\end{align}
This, identity \eqref{capacity.bad.periodic.1} and \eqref{capacity.bad.periodic.sum1} establish Lemma \ref{l.capacity.average}.
\end{proof}

\section{Proof of Theorem \ref{t.main}, $(b)$}\label{s.poi.d}
In this section we adapting the argument of the previous section to show also Theorem \ref{t.main} in case $(b)$. As mentioned in Subsection \ref{sub.ideas},
the main challenge is related to the construction of a mesoscopic covering $\{ K_{k,z} \}_{z \in N_k}$ that plays the same role of $\{ Q_{k,z} \}_{z \in N_k}$  of Subsection \ref{sub.covering} for Theorem \ref{t.main}. In the present case the random positions of the centres imply that there are configurations (with positive probability) in which some of the spheres $\{ \partial B_{\frac \eps 4}(\eps z)\}_{z\in \Phi}$ intersect the boundary of $\{ Q_{k,z} \}_{z \in N_k}$. This prevents us from appealing to Lemma \ref{Kohn_Vogelius} as condition \eqref{contains.balls} is not satisfied.  

\smallskip

We stress that all the results contained in this section besides hold for any dimension $d \geq 3$. However, in the proof of Theorem \ref{t.main}, $(b)$ we obtain the same decay rate of case $(a)$ only in $d=3$.  In higher dimensions we obtain a slower (but still algebraic) rate. In order to best appreciate this dimensional constraint, in the whole section we work in a general dimension $d \geq 3$. 

\smallskip

Throughout this section we set $\delta$ as in Theorem \ref{t.main} and define the parameters
\begin{align}\label{right.regime.ppp}
k  := \lfloor \eps^{-\frac{2}{d+2}} \rfloor, \ \ \ \kappa:= \frac{2}{(d-1)(d+2)}.
\end{align}

\subsection{Partition of the holes and mesoscopic covering of $D$}\label{sub.covering.ppp}
This subsection contains an adaptation  to the case of random centres of Lemma \ref{l.geometry.periodic} and of the sets $\{ Q_{k,z}\}_{z \in N_k}$. 

\bigskip

\begin{lem}\label{l.geometry.ppp}  Let $\delta$ be as in Theorem \ref{t.main}. We recall the definition \eqref{minimal.distance} of $R_{\eps,z}$.  For $\omega \in \Omega$, we consider a realization of the marked point process $(\Phi; \mathcal{R})$ and of the associated set of holes $H^\eps$. Then, there exists a partition
$$
H^\eps:= H^\varepsilon_{g} \cup H^\varepsilon_{b},
$$
with the following properties:
\begin{itemize}
\item There exists a subset of centres $n^\eps(D) \subset \Phi^\eps(D)$ such that
\begin{align}\label{good.set.ppp}
H^\varepsilon_g : = \bigcup_{z \in n^\varepsilon(D)} B_{\aeps \rho_z}( \varepsilon z ), \ \ \ \ \min_{z \in n^\eps(D)}R_{\eps, z}\geq \eps^{2}, \ \ \ \ \ \max_{z \in n^\eps(D)}\aeps \rho_z \leq \eps^{1+\delta},
\end{align}
and such that
\begin{align}\label{no.overlapping}
2 \sqrt d \aeps \rho_z \leq R_{\eps,z}, \ \ \ \ \text{for every $z \in n^\eps(D)$.} 
\end{align}

\smallskip

\item There exists a set $D^\eps_b(\omega) \subset  \{ x \in \R^d \, \colon \, \mathop{dist}(x, D) \leq 2\}$ satisfying
\begin{align}
H^\varepsilon_{b} \subset D^\eps_b, \ \ \ \capacity ( H^\varepsilon_b, D_b^\eps) \lesssim \aeps \sum_{z \in \Phi^\eps(D) \backslash n^\eps(D)} \rho_z^{d-2} \label{capacity.sum}
\end{align}
and for which
\begin{align}\label{ppp.distance.good.bad}
B_{R_{\eps,z}}(\eps z) \cap D^\eps_b = \emptyset, \ \ \  \ \ \ \ \text{for every $z \in n^\eps(D)$.} 
\end{align}
\end{itemize}
\end{lem}

\begin{proof}[Proof of Lemma \ref{l.geometry.ppp}] 
The construction for the sets $H^\eps_g, H^\eps_b$ and $D^\eps_b$ is very similar to the one implemented in the proof of Lemma \ref{l.geometry.periodic} and in the proof of  \cite[Lemma 4.2]{GHV}. Also in this case, we denote by $I_\eps^b \subset \Phi_\eps(D)$ the set that generates the holes $H_b^\eps$. We construct $I^\eps_b$ in the following way: As in Lemma \ref{l.geometry.periodic}, we include in it the points $z \in \Phi^\eps(D)$ whose mark $\rho_z$ is bigger than the threshold $\eps^{-\frac{2}{d-2} +\delta}$, namely 
\begin{align}\label{index.set.ppp.1}
J^\eps_b= \Bigl\{ z \in \Phi^\eps(D) \colon \aeps\rho_z \geq \eps^{-\frac{2}{d-2} +\delta}\Bigr\}.
\end{align}
Contrarily to the periodic case of Lemma \ref{l.geometry.periodic}, we also need to consider the points of $\Phi_\eps(D)$ which are very close to each other: We indeed define
\begin{align}\label{index.set.ppp.2}
K^{\eps}_{b}:= \biggl\{ z \in \Phi^\eps(D) \backslash J^\eps_b \, \colon \, R_{\eps,z} \leq \eps^2 \biggr\}.
\end{align}
We now include in the set $I^\eps_b$ also those points that are close when compared to their radii, i.e. the set
\begin{align}\label{index.set.ppp.4}
C^\eps_b:=  \Bigl\{ z \in \Phi^\eps(D) \backslash (J^\eps_b \cup K^\eps_b) \,  \colon \,  2\sqrt d \aeps\rho_z \geq R_{\eps,z} \Bigr\}.
\end{align} 
Finally, given the holes
\begin{align*}
\tilde H^\eps_b := \bigcup_{z \in  J^\eps_b \cup K^\eps_b \cup C^\eps_b} B_{2 \aeps\rho_z}(\eps z),
\end{align*}
we consider the set of points in $\Phi_\eps(D) \backslash (J^\eps_b \cup K^\eps_b \cup C^\eps_b)$ that are close to the set $\tilde H_b^\eps$, i.e.
\begin{align}\label{index.set.ppp.3}
 \tilde I^\eps_b := \left \{ z \in \Phi^\eps(D) \backslash(J^\eps_b \cup K^\eps_b \cup C^\eps_b) \colon  \tilde H^\eps_b  \cap B_{R_{\eps,z}}(\eps z) \neq \emptyset \right\}.
\end{align}
 We define
\begin{equation}\label{definition.bad.index.ppp}
\begin{aligned}
I^\eps_b := \tilde I^\eps_b \cup J^\eps_b \cup&  K^\eps_b \cup C^\eps_b, \ \ \ n^\eps(D) := \Phi^\eps(D) \backslash I^\eps_b \\
H^\eps_b:=  \bigcup_{z \in I^\eps_b} B_{\aeps\rho_z}(\eps z), \ \ \ H^\eps_g&:=  \bigcup_{z_j \in n^\eps(D)} B_{\aeps \rho_z}(\eps z), \ \ \ D^\eps_b:= \bigcup_{z \in I^\eps_b}  B_{2\aeps\rho_z}(\eps z). 
\end{aligned}
\end{equation}

\bigskip

It remains to show that these sets satisfy properties \eqref{good.set.ppp}- \eqref{ppp.distance.good.bad}. Properties \eqref{good.set.ppp}, \eqref{no.overlapping} are immediate consequences of definitions \eqref{index.set.ppp.1}, \eqref{index.set.ppp.2}, \eqref{index.set.ppp.4} and \eqref{definition.bad.index.ppp}. Properties \eqref{capacity.sum} and  \eqref{ppp.distance.good.bad}  may be proven as \eqref{capacity.sum.periodic} and \eqref{distance.good.bad.periodic} of Lemma \ref{l.geometry.ppp}. The proof of Lemma \ref{l.geometry.ppp} is complete.
\end{proof}

\smallskip

For $k$ as in \eqref{right.regime.ppp}, let $\{ Q_{k,z} \}_{z\in N_k}$ be as in Subsection \ref{sub.covering}. For every $z \in N_k$ we define the sets $N_{k,z}$ as in \eqref{number.covering}. We stress that, in this case, \eqref{number.covering} is ill-defined for the realizations of $\Phi$ such that there are points in $\Phi^\eps_\delta(D)$  that fall on the boundary of the cubes $\{Q_{k,z}\}_{z \in N_k}$. This issue may be easily solved by fixing a deterministic rule to assign these points to a particular cube that shares the boundary considered. We stress that all the arguments of this section do not depend on this rule since the set of the boundaries of the covering  $\{Q_{k,z}\}_{z \in N_k}$ has zero (Lebesgue)-measure.

\smallskip

For $z\in N_k$ and $w \in N_{k,z}$, we define the modification of the minimal distance $R_{\eps,w}$ (see Figure \eqref{covering.pic1}):
\begin{align}\label{def.tilde.R}
\tilde R_{\eps,w}:= \begin{cases}
R_{\eps,w} \ \ \ \ &\text{if $\eps w \in Q_{z, k-1}$}\\
\eps^{1+\kappa} \wedge R_{\eps,w} \ \ \   &\text{if $\mathop{dist}(\eps w; \partial Q_{z,k} ) \leq \eps^{1+\kappa}$}\\
(2^{n-1}\eps^{1+\kappa})\wedge R_{\eps,w}  &\text{if $\eps w \notin Q_{z,k-1}$, \  $2^{n-1}\eps^{1+\kappa} \leq \mathop{dist}(\eps w; \partial Q_{z,k} ) \leq 2^n \eps^{1+\kappa}$}.
 \end{cases}
\end{align}
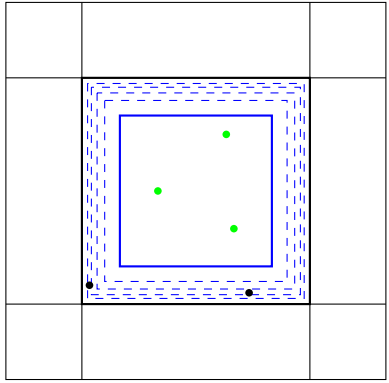
\begin{figure}
\begin{minipage}[c]{7cm}
\begin{tikzpicture}[scale=0.5]
\draw (-4,0) -- (6,0);
\draw (-4,6) -- (6,6);
\draw (-2,8) -- (-2,-2);
\draw (4,8) -- (4,-2);
\draw[thick, black] (-2,6) rectangle (4,0);
\draw[thick, blue] (-1,5) rectangle (3,1);
\draw[dashed, blue] (-1.4,5.4) rectangle (3.4,0.6);
\draw[dashed, blue] (-1.6,5.6) rectangle (3.6,0.4);
\draw[dashed, blue] (-1.75,5.75) rectangle (3.75,0.25);
\draw[dashed, blue] (-1.85,5.85) rectangle (3.85,0.15);

\fill[green] (2,2) circle (0.1);
\fill[green] (0,3) circle (0.1);
\fill[black] (-1.8,0.5) circle (0.1);
\fill[green] (1.8, 4.5) circle (0.1);
\fill[black] (2.4, 0.3) circle (0.1);



\draw (-4,8) rectangle (6,-2);
;\end{tikzpicture}
\end{minipage}
\caption{{\small The square in the thick black line is $Q_{k,z}$, while the blue on is $Q_{k-1,z}$. The dots are a points of $\Phi^\eps_\delta(D)$.  The dashed squares correspond to the sets $\{ x \in Q_{k,z}\, \colon \, \mathop{dist}(x; \partial Q_{k,z}) > 2^n \eps^{1+\kappa} \}$. Inside the blue square (i.e. for the green dots) the random variable $\tilde R_{\eps,w}= R_{\eps,w}$. In each dashed frame (i.e. for the black points), $\tilde R_{\eps, w}= R_{\eps,w} \wedge (2^n \eps^{1+\kappa})$. } }\label{covering.pic1}
\end{figure}
We aim at obtaining a (random) collection of disjoint sets $\{ K_{k,z}\}_{z\in N_k}$ having size $\simeq \eps k$ and such that for every $z \in N_k$ and $w \in \Phi^\eps_\delta(D)$ 
\begin{align}
B_{\tilde R_{\eps,z}}(\eps w) \cap K_{k,z} = \emptyset \ \ \ \  \text{OR} \ \ \  \ B_{2\tilde R_{\eps,z}}(\eps w) \subset K_{k,z}.
\end{align}
We modify $\{Q_{k,z}\}_{z \in N_k}$ as follows: For $\kappa$ as in \eqref{right.regime.ppp}, any $z \in N_k$ and $w \in N_{k,z}$, we consider the cubes
\begin{align}\label{minimal.cube}
\tilde Q_{\eps,w}:= \eps w +2[ - \tilde R_{\eps,z}; \tilde R_{\eps,z}].
\end{align}
Note that, by definition \eqref{minimal.distance}, all the cubes above are disjoint. For every $z\in N_k$, we thus set (see Figure \eqref{covering.pic}) 
\begin{align}\label{covering.Poisson}
K_{k, z} := \bigl(Q_{k, z} \bigcup_{w \in N_{k,z}} Q_{\eps,w} \bigr) \backslash \bigcup_{w \in \Phi_{\delta}^\eps(D) \backslash N_{k,z}} Q_{\eps,w}.
\end{align}
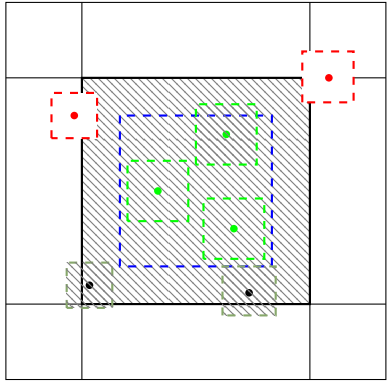
\begin{figure}
\begin{minipage}[c]{7cm}
\begin{tikzpicture}[scale=0.5]
\draw (-4,0) -- (6,0);
\draw (-4,6) -- (6,6);
\draw (-2,8) -- (-2,-2);
\draw (4,8) -- (4,-2);
\draw[pattern=north west lines, pattern color=gray] (-2,6) rectangle (4,0);
\draw[black, thick] (-2,6) rectangle (4,0);
\draw[dashed, thick, blue] (-1,5) rectangle (3,1);
\fill[green] (2,2) circle (0.1);
\fill[green] (0,3) circle (0.1);
\fill[black] (-1.8,0.5) circle (0.1);
\fill[green] (1.8, 4.5) circle (0.1);
\draw[dashed, thick, green] (1.2,1.2) rectangle (2.8,2.8);
\draw[dashed, thick, green] (-0.8, 2.2) rectangle (0.8,3.8);
\draw[dashed, thick, asparagus] (-2.4, -0.1) rectangle (-1.2, 1.1);
\fill[pattern=north west lines, pattern color=gray]  (-2.4, -0.1) rectangle (-1.2, 1.1);
\fill[white] (3.8, 5.35) rectangle (5.15, 6.7);
\draw[dashed, thick, red] (3.8, 5.35) rectangle (5.15, 6.7);
\fill[red] (4.5,6) circle (0.1);
\fill[pattern=north west lines, pattern color=gray] (1, 3.7) rectangle (2.6, 5.3);
\draw[dashed, thick, green] (1, 3.7) rectangle (2.6, 5.3);
\fill[white] (-2.8,4.4) rectangle (-1.6, 5.6);
\draw[dashed, thick, red] (-2.8,4.4) rectangle (-1.6, 5.6);
\fill[red] (-2.2, 5) circle (0.1);
\fill[black] (2.4, 0.3) circle (0.1);
\fill[pattern=north west lines, pattern color=gray]  (1.7, -0.3) rectangle (3.1, 1);
\draw[dashed, thick, asparagus]  (1.7, -0.3) rectangle (3.1, 1);



\draw (-4,8) rectangle (6,-2);
;\end{tikzpicture}
\end{minipage}
\caption{{\small The construction of $K_{\eps,z}$ from the cube $Q_{k,\eps}$. The dashed grey area corresponds to the set $K_{\eps,z}$, while $Q_{\eps,z}$ is the square bounded by the thick black line. The green dots are the points of $\Phi^\eps_\delta$ that fall inside the set $Q_{k-1,z}$ (here bounded by the dashed blue line). The red dots are the points that are outside of $Q_{k,z}$ but whose associated cube intersects $\partial Q_{k,z}$.  The black dots are the points that are in $Q_{k,z} \backslash Q_{k-1,z}$. Note that the cubes associated to the black and red dots are typically smaller than the ones associated to the green dots due to the cut-off $\tilde R_{\eps,z}$.}}\label{covering.pic}
\end{figure}
Since the cubes $\{ Q_{\eps,z} \}_{z \in \Phi_\delta^\eps(D)}$ are disjoint we have that
\begin{equation}
\begin{aligned}\label{properties.covering}
&\bigcup_{z \in N_k} K_{k,z} \supseteq D, \ \ \  |\mathop{diam}(K_{k,z})| \lesssim k \eps,\\
& ( k  - \eps^{\kappa})^d \eps^d \leq |K_{k,z}| \leq ( k + \eps^{\kappa})^d \eps^d.
\end{aligned}
\end{equation}
We emphasize that the previous properties hold for every realization of the point process $\Phi$. The introduction of the modified random variable $\tilde R_{\eps,z}$ is needed to ensure that the second property in \eqref{properties.covering} holds with $\eps^\kappa$ instead of $1$. This yields that the difference between the volume of the set $K_{k,z}$ and the cube $Q_{k,z}$ is of order $\eps^{d +\kappa} k^{d-1}$  instead of $\eps^d k^{d-1}$. This condition plays a crucial role in the proof of the theorem (see \eqref{main.5.ppp.2}) and is the main term that forces the dimensional constraint $d=3$ in the rates of convergence.

\subsection{Quenched estimates for the homogenization error}
In this section we adapt Lemma \ref{det.estimate} to the current setting. As in the case of Lemma \ref{det.estimate}, the next result relies on a variation of Lemma \ref{det.KM} that allows us to replace  in the definition \eqref{def.mu.eps} of $\mu_\eps$ the radii $\frac{\eps}{4}$ with $\tilde R_{\eps,z}$ defined in \eqref{def.tilde.R}.

\smallskip
 
We define the oscillating test function $w_\eps \in H^1(D)$ as done in Subsection \ref{sub.quenched.periodic}, this time using the sets $H^\eps_b, H^\eps_g$ and $D^\eps_b$ of Lemma \ref{l.geometry.ppp} with $\delta$ as in Theorem \ref{t.main}, and $R_{\eps,z}$ instead of $\frac \eps 4$ in \eqref{oscillating.good}. We also define the analogues of \eqref{averaged.sum},  this time associated to the covering $\{ K_{k,z} \}_{z \in N_k}$ constructed in the previous subsection: For every $z \in N_k$ we indeed set
\begin{align}\label{averaged.sum.ppp}
S_{k,z}:= \frac{\eps^d}{|K_{k,z}|}\sum_{w \in N_{k,z}}Y_{\eps, w}, \ \ \ \ \ Y_{\eps,w}:= \rho_w^{d-2} \frac{\tilde R_{\eps,w}^{d-2}}{\tilde R_{\eps,w}^{d-2}- \eps^d \rho_w^{d-2}}.
\end{align}

\bigskip

\begin{lem}\label{det.estimate.ppp}
Let $W_\eps$ be as in \eqref{corrector} and let $w_\eps$ be defined as above. Then, for every $\eps >0$ and $k \in \N$ such that $\eps k \leq 1$ 
we have that
\begin{equation}
\begin{aligned}
\| u_\varepsilon - W_\varepsilon u \|_{H^1_0(D)}& \lesssim \biggl( (k\eps)^2 \eps^d \sum_{z \in \Phi_{\delta}^\eps(D)} \rho_z^{2(d-2)} \bigl( \frac{\eps}{\tilde R_{\eps,z}}\bigr)^d + \eps^d \sum_{z \in \Phi^\eps(D) \backslash n^\eps(D)} \rho_z^{d-2} \biggr)^{\frac 1 2}\\
& \quad + \biggl( \avsum_{z \in \mathring{N}_k} (S_{k,z} - \lambda \E\bigl[ \rho^{d-2}\bigr])^2 +   (k\eps)^3 \avsum_{z \in N_k \backslash \mathring{N}_{k}} (S_{k,z}  - \lambda \E_\rho \bigl[ \rho^{d-2} \bigr])^2 \biggr)^{\frac 1 2}\\
&\quad  + \biggl( \eps^{d+\frac{2d}{d+2}}\sum_{z \in N_k} \sum_{w \in N_{k,z}, \atop \eps w \in Q_{k,z} \backslash Q_{k-1,z}} \rho_w^{2(d-2)}\biggr)^{\frac 12}.
\end{aligned}
\end{equation}
\end{lem}

\bigskip

\begin{lem}\label{det.KM.ppp}
Let $u_\eps, u$ and $W_\eps$ be as in Lemma \ref{det.estimate} and let $w_\eps$ be as defined above. Then
\begin{align*}
\| u_\varepsilon - W_\varepsilon u \|_{H^1_0(D)}^2& \lesssim \| w_\eps - 1\|_{L^2(D)}^2 + \|\nabla(w_\eps- W_\eps) \|_{L^2(\R^d)}^2\\
& \quad \quad +  \|\nabla(\tilde W_\eps- W_\eps) \|_{L^2(\R^d)}^2 + \| \mu_\eps - C_0 \|_{H^{-1}(D)}^2,
\end{align*}
where $\tilde W_\eps$ is defined as in \eqref{corrector} with $R_{\eps,z}$ substituted by $\tilde R_{\eps,z}$. Furthermore, in this case
\begin{align}\label{def.mu.eps.ppp}
\mu_\eps := \sum_{z \in \Phi_\delta^\eps(D)} \partial_n \tilde w_{\eps,z} \, \delta_{\partial B_{\tilde R_{\eps,z}}(\eps z)},
\end{align}
with $\tilde w_{\eps,z}$ as in \eqref{def.harmonic.annuli} with $\tilde R_{\eps,z}$ instead of $R_{\eps,z}$.
\end{lem}

{
\begin{proof}[Proof of Lemma \ref{det.estimate.ppp}]
 Analogously to the proof of Lemma \ref{det.estimate}, we appeal to Lemma \ref{det.KM.ppp} and reduce to showing that
\begin{align}\label{L.2.norm.ppp}
\| \nabla( W_\eps - w_\eps) \|_{L^2(D)}^2 + \| w_{\eps} - 1 \|_{L^2(D)}^2 &\lesssim  \eps^{d+2} \sum_{z \in n_\eps(D)} \rho_z^{d-2} + \eps^d \sum_{\Phi^\eps(D) \backslash n^\eps(D)} \rho_z^{d-2},\\
  \|\nabla(\tilde W_\eps- W_\eps) \|_{L^2(\R^d)}^2 &\lesssim \eps^{d + \frac{2d}{d+2}}\sum_{z \in N_k} \sum_{w \in N_{k,z}, \atop \eps w \in Q_{k,z} \backslash Q_{k-1,z}} \rho_w^{2(d-2)} \label{comparison.correctors}
\end{align}
and 
\begin{equation} \label{H.minus.ppp}
\begin{aligned}
 \| \mu_\eps - C_0\|_{H^{-1}(D)}^2  &\lesssim (k\eps)^2 \eps^d \sum_{z \in \Phi_{\delta}^\eps(D)} \rho_z^{2(d-2)}  \bigl( \frac{\eps}{\tilde R_{\eps,z}}\bigr)^d\\
 &\quad +  (k\eps)^3 \avsum_{z \in N_k \backslash \mathring{N}_{k}} (S_{k,z}  - \E_\rho \bigl[ \rho^{d-2} \bigr])^2 +  \avsum_{z \in \mathring{N_k}} (S_{k,z} -  \lambda \E\bigl[ \rho^{d-2}\bigr] \bigr)^2.
  \end{aligned}
 \end{equation}

\smallskip

Inequality \eqref{L.2.norm.ppp} may be argued exactly as done for \eqref{L.2.norm} in the proof of Lemma \ref{det.estimate}, this time appealing to Lemma \ref{l.geometry.ppp} instead of Lemma \ref{l.geometry.periodic}. 

\smallskip

We thus turn to \eqref{comparison.correctors}. We begin by remarking that $\tilde W_{\eps}$ is well-defined: Indeed, by definition \eqref{minimal.distance} and \eqref{def.tilde.R}, we have that $\tilde R_{\eps,z} \leq R_{\eps,z} \leq \frac \eps 4$ for every $z \in \Phi^\eps_\delta(D)$. Furthermore, since $\kappa < \delta$ (c.f. \eqref{right.regime.ppp} and Theorem \ref{t.main}), it follows from \eqref{thinning.psi} that $ 2 \aeps \rho_z \leq \tilde R_{\eps,z}$, for  every $z\in \Phi^\eps_\delta(D)$. Therefore, comparing the two definitions of $W_\eps$ and $\tilde W_\eps$, we use \eqref{def.tilde.R} to bound:
\begin{align}\label{comparison.correctors.1}
\| \nabla( \tilde W_\eps- W_\eps)\|_{L^2}^2 = \sum_{z \in N_k} \sum_{w \in N_{k,w}, \atop \eps w \in Q_{k} \backslash Q_{k-1}}\| \nabla( \tilde W_\eps- W_\eps)\|_{L^2(B_{R_{\eps,w}}(\eps w))}^2.
\end{align}
Since, if $R_{\eps,w} \neq \tilde R_{\eps,w}$, then $\eps^{1+\kappa}\leq \tilde R_{\eps,w} \leq R_{\eps,w}$, we have that
\begin{align}
\| \nabla( \tilde W_\eps- W_\eps)\|_{L^2(B_{R_{\eps,w}}(\eps w)}^2 &\leq \int_{B_{R_{\eps,w}}(\eps w) \backslash B_{\eps^{1+\kappa}}(\eps w)} |\nabla W_\eps|^2 \\
&\quad \quad + \int_{B_{R_{\eps,w}}(\eps w)\backslash B_{\aeps \rho_w}(\eps w)}|\nabla(W_\eps-\tilde W_\eps)|^2.
\end{align}
 Appealing to \eqref{corrector}, \eqref{def.harmonic.annuli} and the adaptation of \eqref{explicit.formula} for both $\tilde W_\eps$ and $W_\eps$, the previous integrals may be bounded by
 \begin{align}
\| \nabla( \tilde W_\eps- W_\eps)\|_{L^2(B_{R_{\eps,w}}(\eps w)}^2 \lesssim \eps^{d} \rho_w^{2(d-2)} \eps^{2-(d-2)\kappa} + \eps^{d}\rho_w^{3(d-2)} \eps^{2(2-(d-2)\kappa)}
\end{align}
Since $w \in \Phi^\eps_\delta(D)$, we have that $\rho_w \leq \eps^{-\frac{2}{d-2}+\delta}$ so that
 \begin{align}
\| \nabla( \tilde W_\eps- W_\eps)\|_{L^2(B_{R_{\eps,w}}(\eps w)}^2& \lesssim \eps^{d} \rho_w^{2(d-2)} \eps^{2-(d-2)\kappa} + \eps^{d}\rho_w^{2(d-2)} \eps^{2-2(d-2)\kappa + (d-2)\delta}\\
& \stackrel{\delta > \kappa}{ \lesssim}  \eps^{d} \rho_w^{2(d-2)} \eps^{2-(d-2)\kappa}.
\end{align}
Inserting this into \eqref{comparison.correctors.1} and appealing to \eqref{right.regime.ppp} for $\kappa$ yields \eqref{comparison.correctors}.

\smallskip

We finally tackle \eqref{H.minus.ppp}: As done for \eqref{H.minus} of Lemma \ref{det.estimate}, we aim at applying Lemma \ref{Kohn_Vogelius}. We thus pick $\mathcal{Z}= \Phi^\eps_\delta(D)$ and $\mathcal{X}= \{ \aeps \rho_z \}_{z \in \mathcal{Z}}, \mathcal{R}:=\{ \tilde R_{\eps,z}\}_{z \in \mathcal{Z}}$. As shown above in the argument for \eqref{comparison.correctors}, condition \eqref{well.defined.cells} is satisfied. Moreover, thanks to \eqref{covering.Poisson}, the collection $\{ K_{k,z} \}_{z \in N_k}$ satisfies \eqref{contains.balls}. Hence, by Lemma \ref{Kohn_Vogelius}, we have that
\begin{align*}
\|  \mu_\eps - m_k \|_{H^{-1}}& \lesssim  \max_{z \in N_k} \bigl(\mathop{diam}(K_{k,z})\bigr) \bigl(\eps^d \sum_{z \in \Phi_{\delta}^\eps(D)} \rho_z^{2(d-2)}\bigl( \frac{\eps}{\tilde R_{\eps,z}}\bigr)^d  \bigr)^{\frac 1 2}\\
&\stackrel{\eqref{properties.covering}}{\lesssim} \eps k \bigl(\eps^d \sum_{z \in \Phi_{\delta}^\eps(D)} \rho_z^{2(d-2)}\bigl( \frac{\eps}{\tilde R_{\eps,z}}\bigr)^d  \bigr)^{\frac 1 2}
\end{align*}
with 
\begin{align*}
m_k \stackrel{\eqref{averaged.sum.ppp}}{ =}  c_d \sum_{z \in N_k} S_{k,z}  \1_{K_{k, z}}.
\end{align*}
By the triangle inequality it thus only remains to control the norm $\| m_k - \mu \|_{H^{-1}(D)}$. Using \eqref{properties.covering}, this may be done exactly as in the proof of Lemma \ref{det.estimate}. The proof of Lemma \ref{det.estimate.ppp} is complete.
\end{proof}

\begin{proof}[Proof of Lemma \ref{det.KM.ppp}]
This lemma may be argued as done for Lemma \ref{det.KM}. The only difference is that,  in \eqref{distributional.laplacian}, we smuggle in $-\Delta \tilde W_\eps$ instead of $-\Delta W_\eps$ and apply the triangle inequality to bound $\| \nabla (\tilde W_\eps- w_\eps )\|_{L^2} \leq \| \nabla (W_\eps- w_\eps )\|_{L^2} + \| \nabla (\tilde W_\eps- W_\eps )\|_{L^2}$.
\end{proof}

\subsection{Annealed estimates (Proof of Theorem \ref{t.main}, $(b)$)}
{As in case $(a)$ the next lemma provides annealed bounds for some of the quantities appearing in the right-hand side of Lemma \ref{det.estimate.ppp}.}
\begin{lem}\label{l.capacity.average.ppp}
Let $n^\eps(D) \subset \Phi^\eps(D)$ the (random) subset constructed in Lemma \ref{l.geometry.ppp}. Then, there exists a constant $C=C(d, \lambda)$ such that
\begin{align}
\E \bigl[  \eps^d \sum_{z \in \Phi^\eps(D) \backslash n^\eps(D)} \rho_z^{d-2} \bigr] \leq C \eps^{(\frac{2}{d-2}-\delta)\beta}.
\end{align}
\end{lem}

\begin{proof}[Proof of Theorem \ref{t.main}, $(b)$]
We recall that $k$ satisfies \eqref{right.regime.ppp}. Combining Lemma \ref{det.estimate.ppp} and Lemma \ref{l.capacity.average.ppp}, we bound
\begin{equation*}
\begin{aligned}
\E \bigl[ \| u_\varepsilon - \tilde W_\varepsilon u \|_{H^1_0(D)}^2 \bigr] &\lesssim  (k\eps)^2  \E \bigl[ \eps^d \sum_{z \in \Phi_{\delta}^\eps(D)} \rho_z^{2(d-2)} \bigl( \frac{\eps}{\tilde R_{\eps,z}}\bigr)^d \bigr]\\
 &\quad \quad +  \E  \bigl[ \avsum_{z \in \mathring{N}_k} (S_{k,z} - \lambda\E \bigl[ \rho^{d-2} \bigr])^2 \bigr] +   (k\eps)^3 \avsum_{z \in N_k \backslash \mathring{N}_{k}} (S_{k,z}  - \lambda \E_\rho \bigl[ \rho^{d-2} \bigr])^2 \\
 &\quad \quad + \eps^{(\frac{2}{d-2}-\delta)\beta} + \eps^{d +\frac{2d}{d+2}} \E \bigl[ \sum_{z \in N_k} \sum_{w \in N_{k,z}, \atop \eps w \in Q_{k,z} \backslash Q_{k-1,z}} \rho_w^{2(d-2)}\bigr]
\end{aligned}
\end{equation*}
As done in the proof of Theorem \ref{t.main}, this also turns into
\begin{equation}\label{main.1.ppp}
\begin{aligned}
\E \bigl[ \| u_\varepsilon - \tilde W_\varepsilon u \|_{H^1_0(D)}^2 \bigr] &\lesssim  (k\eps)^2  \E \bigl[ \eps^d \sum_{z \in \Phi_{\delta}^\eps(D)} \rho_z^{2(d-2)} \bigl( \frac{\eps}{\tilde R_{\eps,z}}\bigr)^d \bigr]\\
 &\quad \quad +  \E  \bigl[(S_{k,0} - \lambda\E \bigl[ \rho^{d-2} \bigr])^2 \bigr] +   (k\eps)^3 \avsum_{z \in N_k \backslash \mathring{N}_{k}} (S_{k,z}  - \lambda \E_\rho \bigl[ \rho^{d-2} \bigr])^2 \\
 &\quad \quad +  \eps^{(\frac{2}{d-2}-\delta)\beta}+ \eps^{d +\frac{2d}{d+2}} \E \bigl[ \sum_{z \in N_k} \sum_{w \in N_{k,z}, \atop \eps w \in Q_{k,z} \backslash Q_{k-1,z}} \rho_w^{2(d-2)}\bigr].
\end{aligned}
\end{equation}

\smallskip

We now claim that, thanks to \eqref{right.regime.ppp}, the previous estimate reduces to
\begin{equation}
\begin{aligned}\label{main.2.ppp}
\E \bigl[ \| u_\varepsilon - \tilde W_\varepsilon u \|_{H^1_0(D)}^2 \bigr] &\lesssim \bigl( |\log \eps| (k\eps)^2 + k^{-d}\bigr) \E \bigl[\rho^{2(d-2)}\1_{\rho <  -\frac{2}{d-2} +\delta} \bigr]\\
&\quad  + \eps^{(\frac{2}{d-2}-\delta)\beta} +  k^{-2}\eps^{ \frac{2}{(d+2)(d-1)}}.
\end{aligned}
\end{equation}
If the previous estimate holds, by the choice of $\delta$ and \eqref{right.regime.ppp}, we infer that
\begin{equation*}
\begin{aligned}
\E \bigl[ \| u_\varepsilon - \tilde W_\varepsilon u \|_{H^1_0(D)}^2 \bigr] &\lesssim |\log \eps | \eps^{\frac{2d}{d^2-4} \beta \wedge (d-2)} + \eps^{ \frac{2}{(d+2)}( \frac{2}{d-1} + 2)}.
\end{aligned}
\end{equation*}
If $d \leq 3$, we have that 
\begin{align}\label{dimensional.constraint}
\eps^{ \frac{2}{(d+2)}( \frac{2}{d-1} + 2)} \lesssim |\log \eps | \eps^{\frac{2d}{d^2-4} \beta \wedge (d-2)}, \ \ \ \text{for every $\beta> 0$.}
\end{align}
This establishes Theorem \ref{t.main}, $(b)$.

\bigskip

To conclude the proof, we only need to obtain \eqref{main.2.ppp} from \eqref{main.1.ppp}. Arguing as for \eqref{main.4} in the proof of Theorem \ref{t.main}, $(a)$, we reduce to
\begin{equation}\label{main.3.ppp}
\begin{aligned}
\E \bigl[ \| u_\varepsilon - \tilde W_\varepsilon u \|_{H^1_0(D)}^2 \bigr] &\lesssim  (k\eps)^2  \E \bigl[ \eps^d \sum_{z \in \Phi_{\delta}^\eps(D)} \rho_z^{2(d-2)} \bigl( \frac{\eps}{\tilde R_{\eps,z}}\bigr)^d \bigr] +  \E  \bigl[(S_{k,0} - \lambda\E \bigl[ \rho^{d-2} \bigr])^2 \bigr]\\
&\quad\quad  +  \eps^{(\frac{2}{d-2}-\delta)\beta}  +  \eps^{d +\frac{2d}{d+2}} \E \bigl[ \sum_{z \in N_k} \sum_{w \in N_{k,z}, \atop \eps w \in Q_{k,z} \backslash Q_{k-1,z}} \rho_w^{2(d-2)}\bigr].
\end{aligned}
\end{equation}
This implies inequality \eqref{main.2.ppp} provided that
\begin{align}\label{main.comparison}
 \eps^{d +\frac{2d}{d+2}} \E \bigl[ \sum_{z \in N_k} \sum_{w \in N_{k,z}, \atop \eps w \in Q_{k,z} \backslash Q_{k-1,z}} \rho_w^{2(d-2)}\bigr]\lesssim (\eps k)^2 \E\bigl[\rho^{2(d-2)} \1_{\rho < \eps^{-\frac{2}{d-2} +\delta}} \bigr],
\end{align}
\begin{align}\label{main.4.ppp}
 \E \bigl[ \eps^d \sum_{z \in \Phi_{\delta}^\eps(D)} \rho_z^{2(d-2)} \bigl( \frac{\eps}{\tilde R_{\eps,z}}\bigr)^d \bigr] \lesssim |\log \eps| \E \bigl[\rho^{2(d-2)}\1_{\rho <  -\frac{2}{d-2} +\delta} \bigr]
\end{align}
and
\begin{align}\label{main.5.ppp.2}
\E  \bigl[(S_{k,0} - \lambda\E \bigl[ \rho \bigr])^2 \bigr]& \lesssim k^{-d}  \E\bigl[ \rho^2 \1_{\rho < \eps^{-\frac{2}{d-2} +\delta}}\bigr]+ \eps^{(\frac{2}{d-2}-\delta)\beta} +  k^{-2}\eps^{\frac{4}{(d+2)(d-1)}}.
\end{align}

\smallskip

We argue \eqref{main.4.ppp}: Recalling the definition of the covering $\{ Q_{k,z} \}_{z \in N_k}$, we decompose
$$
D \subset \bigcup_{z \in N_k} \bigl( Q_{k-1, z} \cup (Q_{k,z} \backslash Q_{k-1,z})\bigr)
$$
and rewrite
\begin{align}
 \E \bigl[ \eps^d \sum_{z \in \Phi_{\delta}^\eps(D)}& \rho_z^{2(d-2)} \bigl( \frac{\eps}{\tilde R_{\eps,z}}\bigr)^d \bigr]\\
 & = \E \bigl[ \eps^d \sum_{w \in N_k}\biggl( \sum_{z \in \Phi_{\delta}^\eps(D), \atop \eps z \in Q_{k-1,w}} \rho_z^{2(d-2)} \bigl( \frac{\eps}{\tilde R_{\eps,z}}\bigr)^d + \sum_{z \in \Phi_{\delta}^\eps(D), \atop \eps z \in Q_{k,w}\backslash Q_{k-1,w}} \rho_z^{2(d-2)} \bigl( \frac{\eps}{\tilde R_{\eps,z}}\bigr)^d  \biggr) \bigr] .
 \end{align}
Since the process $(\Phi , \mathcal{R})$ is stationary, we bound
 \begin{equation}
 \begin{aligned}\label{main.5.ppp}
 \E \bigl[ \eps^d \sum_{z \in \Phi_{\delta}^\eps(D)}& \rho_z^{2(d-2)} \bigl( \frac{\eps}{\tilde R_{\eps,z}}\bigr)^d \bigr]\\
 &\stackrel{\eqref{number.N.k}}{\lesssim} k^{-d}\E \bigl[  \sum_{z \in \Phi_{\delta}^\eps(D), \atop \eps z \in Q_{k-1,0}} \rho_z^{2(d-2)} \bigl( \frac{\eps}{\tilde R_{\eps,z}}\bigr)^d + \sum_{z \in \Phi_{\delta}^\eps(D), \atop \eps z \in Q_{k,0}\backslash Q_{k-1,0}} \rho_z^{2(d-2)} \bigl( \frac{\eps}{\tilde R_{\eps,z}}\bigr)^d   \bigr].
 \end{aligned}
 \end{equation}
 Let us partition the cube $Q_{k,0}$ into $k^{d}$ cubes of size $\eps$ and let $Q$ be as in \eqref{cubes.meso}; the definitions of $\Phi^\eps_\delta(D)$ and $\tilde R_{\eps,z}$ (c.f. \eqref{thinning.psi}, \eqref{def.tilde.R}) and the stationarity of $(\Phi, \mathcal{R})$ imply that
\begin{align}\label{example.product}
k^{-d}\E \bigl[  \sum_{z \in \Phi_{\delta}^\eps(D), \atop \eps z \in Q_{k-1,0}} \rho_z^{2(d-2)} \bigl( \frac{\eps}{\tilde R_{\eps,z}}\bigr)^d \bigr] \lesssim 
 \E \bigl[\sum_{z \in \Phi(Q)} \rho_z^{2(d-2)} \1_{\rho_z \leq \eps^{-\frac{2}{d-2} +\delta}} \1_{R_{\eps,z} \geq \eps^2} \bigl( \frac{\eps}{R_{\eps,z}}\bigr)^d \bigr].
\end{align}
We now apply Lemma \ref{product.measure} with $G((x,\rho); \omega) =  \bigl(\frac{\eps}{R_{\eps,x}}\bigr)^d \1_{R_{x,\eps} \geq \eps^2} \rho^{2(d-2)} \1_{\rho<\eps^{-\frac{2}{d-2} +\delta}}$ to infer that
\begin{align}
\E \bigl[ \eps^d \sum_{z \in \Phi_{\delta}^\eps(D)} \rho_z^{2(d-2)} \bigl( \frac{\eps}{\tilde R_{\eps,z}}\bigr)^d \bigr] \lesssim \E_\rho \bigl[ \rho_z^{2(d-2)} \1_{\rho_z \leq \eps^{-\frac{2}{d-2} +\delta}} \bigr] \E_\Phi \bigl[ \bigl( \frac{\eps}{R_{\eps,0}}\bigr)^d \1_{R_{\eps,0} \geq \eps^2} \bigr].
\end{align}
By definition of $R_{\eps,z}$ (see \eqref{minimal.distance}) it follows from the properties of the Poisson point process that
\begin{align}\label{log.term}
\E_\phi \bigl[ \1_{R_{\eps,0} > \eps^2}\bigl( \frac{\eps}{R_{\eps,0}}\bigr)^d \bigr] \lesssim |\log \eps|.
\end{align}
Hence, 
$$
k^{-d}\E \bigl[  \sum_{z \in \Phi_{\delta}^\eps(D), \atop \eps z \in Q_{k-1,0}} \rho_z^{2(d-2)} \bigl( \frac{\eps}{\tilde R_{\eps,z}}\bigr)^d \bigr] \lesssim  |\log \eps| \E_\rho \bigl[ \rho_z^{2(d-2)} \1_{\rho_z \leq \eps^{-\frac{2}{d-2} +\delta}} \bigr] 
$$
and \eqref{main.5.ppp} turns into 
 \begin{align}
 \E \bigl[ \eps^d \sum_{z \in \Phi_{\delta}^\eps(D)}& \rho_z^{2(d-2)} \bigl( \frac{\eps}{\tilde R_{\eps,z}}\bigr)^d \bigr]\\
 & \lesssim |\log\eps| \E_\rho \bigl[ \rho_z^{2(d-2)} \1_{\rho_z \leq \eps^{-\frac{2}{d-2} +\delta}} \bigr]  + k^{-d}\E\bigl[ \sum_{z \in \Phi_{\delta}^\eps(D), \atop \eps z \in Q_{k,0}\backslash Q_{k-1,0}} \rho_z^{2(d-2)} \bigl( \frac{\eps}{\tilde R_{\eps,z}}\bigr)^d   \bigr].
 \end{align}

\smallskip

We now tackle the remaining term in the inequality above and claim that, thanks to \eqref{right.regime.ppp},  we have that
\begin{align}\label{main.6.ppp}
 k^{-d}\E\bigl[ \sum_{z \in \Phi_{\delta}^\eps(D), \atop \eps z \in Q_{k,0}\backslash Q_{k-1,0}} \rho_z^{2(d-2)} \bigl( \frac{\eps}{\tilde R_{\eps,z}}\bigr)^d \bigr] \lesssim  \E_\rho \bigl[ \rho_z^{2(d-2)} \1_{\rho_z \leq \eps^{-\frac{2}{d-2} +\delta}} \bigr].
\end{align}
Let $Q_{r}$ be the cube of size $r>0$ centred at the origin. Using \eqref{thinning.psi}, indeed, we bound
\begin{align}
k^{-d}\E\bigl[ \sum_{z \in \Phi_{\delta}^\eps(D), \atop \eps z \in Q_{k,0}\backslash Q_{k-1,0}} \rho_z^{2(d-2)} \bigl( \frac{\eps}{\tilde R_{\eps,z}}\bigr)^d   \bigr] \leq k^{-d}\E\bigl[ \sum_{z \in \Phi( Q_{k}\backslash Q_{k-1})} \rho_z^{2(d-2)} \bigl( \frac{\eps}{\tilde R_{\eps,z}}\bigr)^d \1_{R_{\eps,z} \geq \eps^2} \1_{\rho_z < \eps^{-\frac{2}{d-2} + \delta}} \bigr].
\end{align}
Since we may decompose  the set $Q_k \backslash Q_{k-1}$ into $\lesssim k^{d-1}$ unitary cubes, we use again the stationarity of $(\Phi; \mathcal{R})$ and infer that
\begin{align}\label{logarithmic.1}
k^{-d}\E\bigl[ \sum_{z \in \Phi_{\delta}^\eps(D), \atop \eps z \in Q_{k,0}\backslash Q_{k-1,0}} &\rho_z^{2(d-2)} \bigl( \frac{\eps}{\tilde R_{\eps,z}}\bigr)^d   \bigr]\\
& \leq k^{-1} \E\bigl[ \sum_{z \in \Phi(Q_1)} \rho_z^{2(d-2)} \bigl( \frac{\eps}{\tilde R_{\eps,z}}\bigr)^d \1_{R_{\eps,z} \geq \eps^2} \1_{\rho_z < \eps^{-\frac{2}{d-2} + \delta}} \bigr],
\end{align}
where $Q_1$ is any unitary cube that is contained in $Q_k \backslash Q_{k-1}$\footnote{In principle, in this last step one should distinguish between unitary cubes according to the number of faces that they share with the boundary. However, the argument shown below for the term associated to $Q_1$ may be easily adapted to any of the previous cubes. }.  We now decompose 
\begin{align}
Q_1 &= \sum_{n=1}^{ \lceil -\kappa\log\eps \rceil}A_n,\\
 A_n:= \{ x \in Q_1 \, \colon \, 2^n \eps^\kappa \leq \mathop{dist}(x ; \partial Q_k ) &\leq 2^{n+1}\eps^\kappa \}, \ \ \ \ \ A_0:= \{ x \in Q_1 \, \colon \, \mathop{dist}(x ; \partial Q_k ) \leq \eps^\kappa \}.
\end{align}  
and use \eqref{def.tilde.R} to rewrite
\begin{align}
\E\bigl[ \sum_{z \in \Phi(Q_1)} \rho_z^{2(d-2)} \bigl( \frac{\eps}{\tilde R_{\eps,z}}\bigr)^d &\1_{R_{\eps,z} \geq \eps^2} \1_{\rho_z < \eps^{-\frac{2}{d-2} + \delta}} \bigr]\\
&\lesssim \sum_{n=0}^{ \lceil -\kappa\log\eps \rceil} \E\bigl[ \sum_{z \in \Phi(A_n)} \rho_z^{2(d-2)} \bigl( \frac{\eps}{ R_{\eps,z} \wedge 2^n\eps^{1+\kappa}}\bigr)^d \1_{R_{\eps,z} \geq \eps^2} \1_{\rho_z < \eps^{-\frac{2}{d-2} + \delta}} \bigr]
\end{align}
We now appeal again to Lemma \ref{product.measure} as for \eqref{example.product} to reduce to
\begin{align}
\E\bigl[ \sum_{z \in \Phi(Q_1)} \rho_z^{2(d-2)} \bigl( \frac{\eps}{\tilde R_{\eps,z}}\bigr)^d &\1_{R_{\eps,z} \geq \eps^2} \1_{\rho_z < \eps^{-\frac{2}{d-2} + \delta}} \bigr]\\
&\lesssim \sum_{n=0}^{ \lceil -\kappa\log\eps \rceil} \E_\rho\bigl[  \rho^{2(d-2)} \1_{\rho < \eps^{-\frac{2}{d-2} + \delta}}\bigr] |A_n| \E\bigl[ \bigl( \frac{\eps}{ R_{\eps,0} \wedge 2^n\eps^{1+\kappa}}\bigr)^d \1_{R_{\eps, 0} \geq \eps^2}  \bigr].
\end{align}
Arguing as for \eqref{log.term} and using the stationarity of $\Phi$ we infer that
\begin{align}
\E\bigl[ \sum_{z \in \Phi(Q_1)} \rho_z^{2(d-2)} \bigl( \frac{\eps}{\tilde R_{\eps,z}}\bigr)^d &\1_{R_{\eps,z} \geq \eps^2} \1_{\rho_z < \eps^{-\frac{2}{d-2} + \delta}} \bigr]\\
&\lesssim \E_\rho\bigl[  \rho^{2(d-2)} \1_{\rho < \eps^{-\frac{2}{d-2} + \delta}}\bigr]  \sum_{n=0}^{ \lceil -\kappa\log\eps \rceil} 2^n \eps^{\kappa} \bigl(2^{-dn} \eps^{-d\kappa} - \log \eps \bigr)\\
&\lesssim  \E_\rho\bigl[  \rho^{2(d-2)} \1_{\rho < \eps^{-\frac{2}{d-2} + \delta}}\bigr]  \eps^{-(d-1)\kappa}.
\end{align}
To establish \eqref{main.6.ppp} it only remains to combine the previous inequality with \eqref{logarithmic.1} and use \eqref{right.regime.ppp}. The proof of  \eqref{main.4.ppp} is therefore complete.

\bigskip

Inequality \eqref{main.comparison} may be obtained in a similar way as to that of \eqref{main.4.ppp}: Since we may decompose the set $\bigcup_{z \in N_k} Q_{k,z}\backslash Q_{k-1,z}$ into $n \lesssim (\eps k)^{-d} k^{d-1}$ disjoint cubes $\{ Q_{\eps,i} \}_{i=1}^n$ of size $\eps$, we use  definition \eqref{thinning.psi} and the stationarity of $(\Phi, \mathcal{R})$ to bound
\begin{align}
\eps^{d + \frac{2d}{d+2}} \E \bigl[ \sum_{z \in N_k} \sum_{w \in N_{k,z}, \atop \eps w \in Q_{k,z} \backslash Q_{k-1,z}} \rho_w^{2(d-2)}\bigr] \lesssim k^{-1}\eps^{\frac{2d}{d+2}}\E\bigl[ \sum_{w \in \Phi(Q)} \rho^{2(d-2)} \1_{\rho < \eps^{-\frac{2}{d-2} +\delta}} \bigr]
\end{align}
so that, again by Lemma \ref{product.measure}, we obtain
\begin{align}
\eps^{d + \frac{2d}{d+2}} \E \bigl[ \sum_{z \in N_k} \sum_{w \in N_{k,z}, \atop \eps w \in Q_{k,z} \backslash Q_{k-1,z}} \rho_w^{2(d-2)}\bigr] \lesssim k^{-1}\eps^{\frac{2d}{d+2}}\E\bigl[ \rho^{2(d-2)} \1_{\rho < \eps^{-\frac{2}{d-2} +\delta}} \bigr]
\end{align}
We establish \eqref{main.comparison} after observing that, thanks to \eqref{right.regime.ppp}, it holds $k^{-1}\eps^{\frac{2d}{d+2}} \leq (\eps k)^2$.

\bigskip

We now tackle \eqref{main.5.ppp.2}: By construction (see definition \eqref{covering.Poisson}), the (random) set $K_{\eps,0}$ satisfies
\begin{align}
\bigl\{ w \in \Phi^\eps_\delta(D) \, \colon \, \eps w \in K_{k,0} \bigr\}= \bigl\{ w \in \Phi^\eps_\delta(D) \, \colon \, \eps w \in Q_{k,0} \bigr\} = \Phi^\eps_\delta(D) \cap \Phi(Q_k),
\end{align}
where $Q_k$ is, as above the cube of size $k$ centred at the origin. Hence, decomposing $Q_k= \sum_{i=1}^{k^d}Q_i$ into unitary cubes, definitions \eqref{averaged.sum.ppp} and \eqref{thinning.psi} allow us to rewrite
\begin{align}
S_{k,0} - \lambda\E \bigl[ \rho^{d-2} \bigr] = \frac{\eps^d}{|K_{k,z}|} \sum_{i=1}^{k^d} Z_i -  \lambda\E \bigl[ \rho \bigr]
\end{align}
with
\begin{align}\label{def.Z.2}
Z_i := \sum_{\Phi(Q_{i})} Y_{\eps,z} \1_{\rho < \eps^{-\frac{2}{d-2}+\delta}} \1_{ R_{\eps,z} \geq 2\aeps \rho_z}, \ \ \ i = 1, \cdots, k^d.
\end{align}
We rewrite
\begin{align}
S_{k,0} - \lambda\E \bigl[ \rho^{d-2} \bigr] = \frac{\eps^d}{|K_{k,z}|} \sum_{i=1}^{k^d} (Z_i -  \lambda\E \bigl[ \rho^{d-2} \bigr]) + \lambda (\frac{\eps^dk^d}{|K_{k,z}|} -1)\E_\rho \bigl[ \rho^{d-2} \bigr]
\end{align}
so that the triangle inequality, assumption \eqref{integrability.radii} and the quenched bounds in \eqref{properties.covering} yield
\begin{align}
(S_{k,0} - \lambda\E \bigl[ \rho^{d-2} \bigr])^2 \lesssim \bigl(\avsum_{i=1}^{k^d} (Z_i -  \lambda\E \bigl[ \rho^{d-2} \bigr]) \bigr)^2 + k^{-2}\eps^{2\kappa}.
\end{align}
Appealing to definitions \eqref{def.Z.2}, \eqref{averaged.sum.ppp} and \eqref{minimal.distance}, we observe that $Z_i$ and $Z_j$ are independent whenever $i, j$ are such that $Q_{j}$ and $Q_{i}$ are not adjacent. Hence, by taking the expectation in the previous inequality, we estimate
\begin{equation}
\begin{aligned}\label{main.6.ppp.c}
\E \biggl[ (S_{k,0} - \lambda\E \bigl[ \rho \bigr])^2\biggr] &\lesssim  k^{-2d} \sum_{i=1}^{k^d} \sum_{j \colon Q_{j}, Q_{i} \text{\, adjacent}}  \E\biggl[ (Z_i -  \lambda\E \bigl[ \rho^{d-2} \bigr])  (Z_j -  \lambda\E \bigl[ \rho^{d-2} \bigr]) \biggr]\\
&+ k^{-d}\sum_{i=1}^{k^d}  \E\biggl[ Z_i -  \lambda\E \bigl[ \rho^{d-2} \bigr] \biggr]^2 +  k^{-2}\eps^{2\kappa}.
\end{aligned}
\end{equation}
To establish \eqref{main.5.ppp.2} from \eqref{main.6.ppp.c} it suffices to bound
\begin{align}\label{LLN.1}
\E\biggl[ (Z_i -  \lambda\E \bigl[ \rho^{d-2} \bigr])  (Z_j -  \lambda\E \bigl[ \rho^{d-2} \bigr]) \biggr] &\lesssim \E\bigl[ \rho^{2(d-2)} \1_{\rho < \eps^{-\frac{2}{d-2}+\delta}}\bigr],\\
 \E\biggl[ Z_i -  \lambda\E \bigl[ \rho^{d-2} \bigr] \biggr]  \lesssim  \eps^{(\frac{2}{d-2}-\delta)\beta}. \label{LLN.2}
\end{align}

\bigskip

Inequality \eqref{LLN.1} immediately follows from Cauchy-Schwarz's inequality, the triangle inequality and definitions \eqref{def.Z.2} and \eqref{averaged.sum.ppp}. Again by \eqref{def.Z.2} and \eqref{averaged.sum.ppp}, for every $i=1, \cdots, k^d$, we have that
\begin{align}
 \E\biggl[ Z_i -  \lambda\E \bigl[ \rho^{d-2} \bigr] \biggr] &= {\E\bigl[  \sum_{\Phi(Q)} \rho_z^{d-2} \1_{\rho_z < \eps^{-\frac{2}{d-2}+\delta}} \1_{R_{\eps,z} \geq \eps^2 \vee \aeps \rho_z}\bigr]}-  \lambda\E \bigl[ \rho^{d-2} \bigr]\\
&\quad\quad   + \E\bigl[  \sum_{\Phi(Q)}\eps^d \frac{\rho_z^{2(d-2)}}{\tilde R_{\eps,z}^{d-2} - \eps^d \rho_z^{d-2}} \1_{\rho_z < \eps^{-\frac{2}{d-2}+\delta}}   \1_{R_{\eps,z} \geq \eps^2 \vee \aeps \rho_z}\bigr].
\end{align}
Observing that 
$$
\E \bigl[ \sum_{\Phi(Q)} \rho_z^{d-2} \bigr] = \lambda\E \bigl[ \rho^{d-2} \bigr]
$$
and writing
\begin{align}
\1_{\rho < \eps^{-\frac{2}{d-2}+\delta}}  \1_{R_{\eps,z} \geq \eps^2 \vee \aeps \rho_z} &= 1- \1_{\rho_z \geq \eps^{-\frac{2}{d-2}+\delta}} -   \1_{\rho_z \leq \eps^{-\frac{2}{d-2}+\delta}} \1_{R_{\eps,z} \leq 2\aeps\rho_z \vee \eps^2},
\end{align}
we infer that
\begin{align}
| \E\biggl[ Z_i -  \lambda\E \bigl[ \rho^{d-2} \bigr] \biggr] | & \lesssim \E\bigl[  \sum_{\Phi (Q)} \rho_z^{d-2}  \1_{\rho_z < \eps^{-\frac{2}{d-2}+\delta}} \1_{R_{\eps,z} \leq \aeps \rho \vee \eps^2} \bigr]+ \E\bigl[\rho^{d-2} \1_{\rho \geq \eps^{-\frac{2}{d-2}+\delta}}\bigr]\\
&\quad\quad   + \eps^d \E\bigl[  \sum_{\Phi(Q)} \biggl(\frac{\rho_z^{2}}{\tilde R_{\eps,z}}\biggr)^{d-2} \1_{\rho < \eps^{-2+\delta}} \1_{R_{\eps,z} \geq \aeps \rho_z \vee \eps^2}\bigr].
\end{align}
We now appeal to Lemma \ref{product.measure} with $G((x,\rho); \omega)= \rho \1_{\rho \leq \eps^{-2+\delta}} \1_{R_{\eps,x}\leq \eps^2}$ and to the properties of Poisson point processes to infer that
{\begin{align}\label{estimate.ball.radii}
\E\bigl[  \sum_{\Phi(Q)} \rho_z^{d-2} \1_{\rho_z < \eps^{-2+\delta}} \1_{R_{\eps,z} \leq  \aeps \rho_z \vee \eps^2} \bigr]&\lesssim \E_\rho\biggl[\rho^{d-2} \E_{\Phi}\bigl[ \1_{R_{\eps,0} < \aeps\rho \vee \eps^2} \bigr]  \bigr] \biggr]\\
\stackrel{\eqref{integrability.radii}}{\lesssim}  \eps^d + \eps^{(\frac{2}{d-2} - \delta)\beta}.
\end{align}}
Hence,
\begin{align}
| \E\bigl[ (Z_i -  \lambda\E \bigl[ \rho \bigr] \bigr] | & \lesssim \eps^{d} + \E\bigl[\rho \1_{\rho \geq \eps^{-\frac{2}{d-2}+\delta}}\bigr]   + \E\bigl[  \sum_{\Phi(Q_{i})}\frac{\rho_z^{2(d-2)}}{\tilde R_{\eps,z} - \aeps\rho_z} \1_{\rho < \eps^{-\frac{2}{d-2}+\delta}} \1_{R_{\eps,z} \geq \aeps \rho_z\vee \eps^2 }\bigr]\\
&\stackrel{\eqref{Cheb}}{\lesssim}\eps^{(\frac{2}{d-2}-\delta)\beta} + \E\bigl[  \sum_{\Phi(Q_{i})} \biggl(\frac{\rho_z^{2}}{\tilde R_{\eps,z}}\biggr)^{d-2}  \1_{\rho < \eps^{-\frac{2}{d-2}+\delta}} \1_{R_{\eps,z} \geq \aeps \rho_z \vee \eps^2}\bigr]
\end{align}
To establish \eqref{LLN.2} it remains to use \eqref{integrability.radii}, \eqref{def.tilde.R} an argument similar to \eqref{log.term} and \eqref{logarithmic.1} to infer that
\begin{align}
 \E\bigl[  \sum_{\Phi(Q_{i})} \biggl(\frac{\rho_z^{2}}{\tilde R_{\eps,z}}\biggr)^{d-2}  \1_{\rho < \eps^{-\frac{2}{d-2}+\delta}} \1_{R_{\eps,z} \geq \aeps \rho_z \vee \eps^2}\bigr]&\lesssim \eps^{d- 2 + (\frac{2}{d-2} -\delta)\beta +(d-2)\delta} \E_{\Phi}\bigl[ \tilde R_{\eps,0}^{-(d-2)}\1_{R_{\eps,z} \geq \eps^2}\bigr]\\
 & \lesssim \eps^{(\frac{2}{d-2} -\delta)\beta +(d-2)(\delta- \kappa)}.
\end{align}
Since $\kappa<\delta$, this concludes the proof of \eqref{LLN.2} and, in turn, of \eqref{main.5.ppp.2}. The proof of Theorem \ref{t.main} is thus complete.
\end{proof}

\bigskip

\begin{proof}[Proof of Lemma \ref{l.capacity.average.ppp}]
The proof of this lemma follows the same lines of the argument for Lemma \ref{l.capacity.average}. We resort to the construction made in Lemma \ref{l.geometry.ppp} (c.f. \eqref{definition.bad.index}) to decompose 
\begin{align}\label{capacity.bad.2}
 \eps^d \sum_{z \in J^\eps_b} \rho_z^{d-2} +  \eps^d \sum_{z \in  K^\eps_b }\rho_z^{d-2} + \eps^d \sum_{z \in  C^\eps_b} \rho_z^{d-2} +  \eps^d \sum_{z \in \tilde I^\eps_b} \rho_z^{d-2}.
\end{align}
The expectation of the first sum is bounded as follows: We use definition \eqref{index.set.ppp.1} and argue as done for \eqref{main.4.ppp} to reduce to
\begin{align}
\E \bigl[ \eps^d \sum_{z \in J^\eps_b} \rho_z^{d-2} \bigr] \lesssim \E \bigl[ \sum_{z \in \Phi(Q)} \rho_z^{d-2}\1_{\rho_z \geq \eps^{-\frac{2}{d-2} +\delta}}\bigr] \lesssim \E[\rho^{d-2} \1_{\rho> \eps^{-\frac{2}{d-2} +\delta}} \bigr] \stackrel{\eqref{Cheb}}{\lesssim} \eps^{(\frac{2}{d-2}- \delta)\beta}.
\end{align}
The second sum in \eqref{capacity.bad.2} may be treated likewise since, by definition \eqref{index.set.ppp.3}, we have
$$
\eps^d \sum_{z \in K^\eps_b} \rho_z^{d-2} \leq  \eps^d \sum_{z \in \Phi^\eps(D)} \rho_z^{d-2} \1_{R_{\eps,z} \leq \eps^2}.
$$
Similarly, we use \eqref{index.set.ppp.4} to rewrite
$$
 \eps^d \sum_{z \in C^\eps_b} \rho_z^{d-2} \leq  \eps^d \sum_{z \in \Phi^\eps(D)} \rho_z^{d-2} \1_{\rho_z < \eps^{-\frac{2}{d-2} +\delta}} \1_{R_{\eps,z} \leq 2\aeps\rho_z}.
$$
Taking the expectation, we bound this term by $\eps^{(\frac{2}{d-2}-\delta)\beta}$ as done in \eqref{estimate.ball.radii}. Hence, it only remains to estimate the last sum in \eqref{capacity.bad.2}. As done for the same sum in \eqref{capacity.bad.periodic.1}, we use definition \eqref{index.set.ppp.3} to rewrite
\begin{align}
\E \bigl[ & \eps^d \sum_{z \in \tilde I^\eps_b} \rho_z^{d-2} \bigr]\\
& \lesssim \E \bigl[  \eps^d \sum_{w \in \Phi^\eps(D)} (\1_{\rho_w \geq \eps^{-\frac{2}{d-2} +\delta}} +  \1_{\rho_w < \eps^{-\frac{2}{d-2} +\delta}} \1_{R_{\eps,w} \leq 2\aeps\rho_w \vee \eps^2}) \sum_{z \in \Phi^\eps(D)\backslash\{ w \} \, \atop \eps |w-z| \leq \eps + \aeps\rho_w \wedge 1}  \rho_z^{d-2} \bigr]
\end{align}
and, again by stationarity, reduce to
\begin{align}
\E \bigl[ \eps^d &\sum_{z \in \tilde I^\eps_b} \rho_z^{d-2} \bigr]\\
& \lesssim \E \bigl[  \sum_{w \in \Phi(Q)} (\1_{\rho_w \geq \eps^{-\frac{2}{d-2} +\delta}} +  \1_{\rho_w < \eps^{-\frac{2}{d-2} +\delta}} \1_{R_{\eps,w} \leq 2\aeps\rho_w \vee \eps^2}) \sum_{z \in \Phi^\eps(D)\backslash\{ w \} \, \atop \eps |w-z| \leq \eps + \aeps\rho_w \wedge 1} \rho_z^{d-2} \bigr].
\end{align}
By Lemma \ref{product.measure} applied to
\begin{align}
G( ( x, \rho), \omega)& = (\1_{\rho \geq \eps^{-\frac{2}{d-2} +\delta}} +  \1_{\rho < \eps^{-\frac{2}{d-2} +\delta}} \1_{R_{\eps,x} \leq 2\aeps y \vee \eps^2})  \sum_{z \in \Phi^\eps(D) \, \atop \eps |x-z| \leq \eps + \aeps \rho \wedge 1} \rho_z^{d-2},
\end{align}
we infer that
\begin{align}
\E \bigl[  \eps^d& \sum_{z \in \tilde I^\eps_b} \rho_z^{d-2} \bigr] \lesssim \E_\rho \biggl[ \E\bigl[(\1_{\rho \geq \eps^{-\frac{2}{d-2} +\delta}} +  \1_{\rho < \eps^{-\frac{2}{d-2} +\delta}} \1_{R_{\eps,0} \leq 2\aeps\rho \vee \eps^2})  \sum_{z \in \Phi^\eps(D)\backslash\{ 0\} \, \atop \eps |z| \leq \eps + \aeps\rho \wedge 1} \rho_z^{d-2} \bigr] \biggr].
\end{align}
Since the marks $\{\rho_z\}$ are identically distributed and independent, we use \eqref{integrability.radii} to bound
\begin{align}
\E \bigl[ \eps^d \sum_{z \in \tilde I^\eps_b} \rho_z^{d-2} \bigr] \lesssim  \E_\rho \biggl[ \E_\Phi\bigl[ (\1_{\rho \geq \eps^{-\frac{2}{d-2} +\delta}} &+  \1_{\rho < \eps^{-\frac{2}{d-2} +\delta}} \1_{R_{\eps,0} \leq 2\aeps\rho \vee \eps^2})  \\
&\times  \# \{ z \in \Phi \backslash\{ 0 \} \, \colon \, \eps |z| \leq \eps + \aeps\rho \wedge 1 \} \bigr] \biggr] \, \d x.
\end{align}
Since 
$$
\E_\Phi\bigl[ \# \{ z \in \Phi \backslash\{ 0 \} \,  \colon \, \eps  |z| \leq \eps + \aeps \rho \wedge 1 \} \bigr] \lesssim (\rho^{d-2} + 1),
$$
the first term on the right-hand side above is easily bounded by
\begin{equation}
\begin{aligned}\label{capacity.5}
\E_\rho \biggl[ \1_{\rho \geq \eps^{-\frac{2}{d-2} +\delta}} \E_\Phi\bigl[ \# \{ z \in \Phi \backslash\{ 0 \} \,  \colon \, &\eps  |z| \leq \eps + \aeps \rho \wedge 1 \}  \bigr] \biggr] \\
 & \lesssim \E_\rho\bigl[ \rho^{d-2} \1_{\rho \geq \eps^{-\frac{2}{d-2} +\delta}} \bigr] \stackrel{\eqref{Cheb}}{\lesssim} \eps^{(\frac{2}{d-2} - \delta)\beta}.
 \end{aligned}
 \end{equation}
We now turn to the second term on the right-hand side above: We observe that we may rewrite
\begin{align}
\E_\rho \biggl[ \E_\Phi\bigl[  \1_{\rho < \eps^{-\frac{2}{d-2} +\delta}} \1_{R_{\eps,0} \leq 2\aeps\rho \vee \eps^2} \# \{ z \in \Phi \backslash\{ 0 \} \,  \colon \, \eps  |z| \leq \eps + \aeps \rho \wedge 1 \}  \bigr] \biggr] \\
= \E_\rho \biggl[ \1_{\rho < \eps^{-\frac{2}{d-2} +\delta}} \E_\Phi\bigl[  \1_{R_{\eps,0} \leq 2\aeps\rho \vee \eps^2} \# \{ z \in \Phi \backslash\{ 0 \} \,  \colon \,  |z| \leq 4 \} \bigr] \biggr]
\end{align}
Using H\"older's inequality with exponents  $\frac 3 2$ and $3$ in the inner expectation, definition \eqref{minimal.distance} and the fact that $\Phi$ is a Poisson point process, we thus bound
\begin{align}
\E_\rho& \biggl[ \E_\Phi\bigl[  \1_{\rho < \eps^{-\frac{2}{d-2} +\delta}} \1_{R_{\eps,0} \leq 2\aeps\rho \vee \eps^2} \# \{ z \in \Phi \backslash\{ 0 \} \,  \colon \, \eps  |z| \leq \eps + \aeps \rho \wedge 1 \} \bigr] \biggr] \\
&\leq  \E_\rho \biggl[ \1_{\rho < \eps^{-\frac{2}{d-2} +\delta}} \E_\Phi \bigl[  \1_{R_{\eps,0} \leq 2\aeps\rho \vee \eps^2} \bigr]^{\frac 2 3} \biggr] \stackrel{\eqref{integrability.radii}}{\lesssim} 
\eps^{\frac{2d}{3}} + \eps^{\frac 2 3 (1+ 2\delta + (\frac{2}{d-2} - \delta)\beta)}.
\end{align}
Thanks to the choice of $\delta$ and since $d \geq 2$, the right-hand side is always bounded by $\eps^{ (\frac{2}{d-2} - \delta)\beta}$. Combining this with \eqref{capacity.5} yields
\begin{align}
\E \bigl[ & \eps^d \sum_{z \in \tilde I^\eps_b} \rho_z^{d-2} \bigr] \lesssim \eps^{(\frac{2}{d-2} - \delta)\beta}.
\end{align}
This concludes the proof of Lemma \ref{l.capacity.average.ppp}.
\end{proof}

\bigskip

\section{Auxiliary results}\label{Aux}
Let $\mathcal{Z}:= \{ z_i \}_{i \in I} \subset D$ be a collection of points and let $\mathcal{X}:=\{ X_i \}_{i\in I}, \mathcal{R}:=\{ r_i \}_{i\in I} \subset \R_+$. We assume that 
\begin{align}\label{well.defined.cells}
2X_i < r_i <  \min_{z_j \in \mathcal{Z}, \atop z_j \neq z_i}\{ |z_j - z_i|\}, \ \ \ \text{for every $z_i \in \mathcal{Z}$.}
\end{align}
We define the measure 
\begin{align}\label{measure.M}
M:= \sum_{i \in I} \partial_n v_i \delta_{\partial B_{r_i}(z_i)} \in H^{-1}(D),
\end{align}
where each $v_i \in H^1(B_{r_i}(z_i))$ is the solution of \eqref{def.harmonic.annuli} with $B_{\eps \rho_z}(\eps z)$ and $B_{R_{\eps,z}}(\eps z)$ replaced by $B_{X_i}(z_i)$ and $B_{r_i}(z_i)$, respectively.

\smallskip

The next lemma is a generalization of the result by \cite{Kohn_Vogelius} used in \cite{Kacimi_Murat} to show the analogue of Theorem \ref{t.main} in the case of periodic holes $H\eps$.\begin{lem}\label{Kohn_Vogelius}
Let $\mathcal{Z}$, $\mathcal{X}$ and $\mathcal{R}$ be as above. Then, there exists a constant $C=C(d)$ such that for every finite Lipschitz and (essentially) disjoint covering $\{ K_j\}_{j \in J}$ of $D$ such that
\begin{align}\label{contains.balls}
B_{2r_i}(z_i) \subset K_j \ \ \ \text{OR} \ \ \ B_{r_i}(z_i) \cap K_j= \emptyset \ \ \ \ \text{for every $i \in I$, $j \in J$}
\end{align}
we have that
\begin{align}\label{KV.2}
\| M -m \|_{H^{-1}(D)} \leq C  \max_{j \in J}\mathop{diam}(K_j) \bigl( \sum_{i \in I} X_i^{2(d-2)} r_i^{-d}  \bigr)^{\frac 1 2},
\end{align}
with 
\begin{align}\label{mean.m}
m := c_d \sum_{j \in J} \bigl(\frac{1}{|K_{j}|} \sum_{i \in I, \atop z_i \in K_j}\frac{X_i^{d-2} r_i^{d-2}}{r_i^{d-2} - X_i^{d-2}} \bigr) \1_{K_{j}}.
\end{align}
Here, the constant $c_d$ is as in \eqref{strange.term}.
\end{lem}

\bigskip

The next result is a very easy consequence of the assumptions (i)-(iii) on the marked point process $(\Phi, \mathcal{R})$. Since it is used extensively in the proof of Theorem \ref{t.main}, in the sake of a self-contained presentation, we give below the statement and its brief proof.
\begin{lem}\label{product.measure}
Let $A\subset \R^d$ be a bounded set containing the origin. Let $(\Phi; \mathcal{R})$ satisfy (i)-(iii) with $\Phi=\mathop{Poi}(\lambda)$. For every $\eps > 0$ and $x \in \R^d$, let $R_{\eps,z}$ be as in \eqref{minimal.distance}. Then for every $G:  \R^d \times \R_+ \times \Omega \to \R$ it holds
\begin{align*}
\E \bigl[ \sum_{z \in \Phi(A)} G( (z, \rho_z); \omega &\backslash \{(z, \rho_z)\} ) \bigr]= \lambda |A| \E_\rho \biggl[ \E \bigl[ G((0, \rho); \omega) \bigr] \biggr].
\end{align*}
\end{lem}

\bigskip

\begin{proof}[Proof of Lemma \ref{Kohn_Vogelius}]  With no loss of generality, we give the proof for $d=3$. We start by remarking that, thanks to \eqref{contains.balls}, for every $j \in J$,  there exists $\eta_j \in C^\infty_0(K_j)$ such that $\eta_j =1$ in $\bigcup_{i \in I, \atop z_i \subset K_j} B_{r_i}(z_i)$. This in particular allows us to rewrite the measure $M$ in \eqref{measure.M} as
\begin{align}\label{organizing.M}
M= \sum_{j \in J} \eta_j M_j, \ \ \ M_j:= \sum_{i \in I, \atop z_i \subset K_j} \partial_n v_i \delta_{\partial B_{r_i}(z_i)}
\end{align}
and use the definition of the capacitary functions $\{ v_i \}_{i\in I}$ (see also \eqref{explicit.formula}) to observe that $m$ in \eqref{mean.m} satisfies
\begin{align}\label{organizing.m}
m:= \sum_{j\in J} m_j, \ \ \ m_j :=  \bigl(\frac{1}{|K_j|} \sum_{i \in I, \atop z_i \subset K_j} \int_{\partial B_{r_i}(z_i)} \partial_n v_i \bigr) \1_{K_j}
\end{align}
For every $j\in J$, we thus define  $q_j \in H^1(K_j)$ as the (weak) solution to
\begin{align}\label{def.q.eps}
\begin{cases}
-\Delta q_{j} = 	\eta M_j -  m_{j} \ \ \ &\text{ in $K_{j}$}\\
\partial_n q_{j} = 0 \ \ \ &\text{ on $\partial K_{j}$}
\end{cases}, \ \ \ \  \int_{K_{j}} q_{j}=0,
\end{align}
in the sense that for every $u \in H^1(K_j)$ 
$$
\int_{K_j} \nabla u \cdot \nabla q_j = \langle M_j ; \eta u \rangle - \int_{K_j} m_j u. 
$$
We stress that $q_j$ exists since $K_{j}$ is a Lipschitz domain and, thanks to \eqref{contains.balls} and \eqref{organizing.M}-\eqref{organizing.m}, the compatibility condition
$$
\langle M_j ; \eta \rangle - \int_{K_j} m_j = 0
$$
is satisfied.

\bigskip

By \eqref{def.q.eps} and \eqref{organizing.M}-\eqref{organizing.m}, for any $\phi \in H^1_0(D)$ we thus have that
\begin{align*}
\langle M - m ; \phi \rangle = \sum_{j \in J} \int_{K_{j}} \nabla q_{j} \cdot \nabla \phi,
\end{align*}
and, by Cauchy-Schwarz's inequality, also
\begin{align}\label{H.minus.3}
\| M - m \|_{H^{-1}(D)} \leq  \bigl(\sum_{j \in J}  \int_{K_{j}}|\nabla q_{j}|^2 \bigr)^{\frac 1 2}. 
\end{align}
We now claim that for each $j \in J$
\begin{align}\label{H.minus.5.aux}
\bigl( \int_{K_{j}}|\nabla q_{j}|^2 \bigr)^{\frac 1 2} \lesssim  \mathop{diam}(K_j)  \bigl( \sum_{i \in I, \atop z_i \in K_j} X_i^2 r_i^{-3} \bigr)^{\frac 1 2}.
\end{align}
This inequality and \eqref{H.minus.3} immediately yield the statement of Lemma \ref{Kohn_Vogelius}.

\bigskip

We argue \eqref{H.minus.5.aux} as follows: testing the equation for $q_{j}$ with $q_{j}$ itself and using that $q_j$ has zero mean (see \eqref{def.q.eps}), we obtain
\begin{align*}
 \int_{K_{j}}|\nabla q_{j}|^2 &= \sum_{i \in I, \atop z_i \in K_j} \int_{\partial B_{r_i}(z_i)} \partial_n v_i \, q_{z}.
 \end{align*}
By Cauchy-Schwarz's inequality, this implies that
 \begin{align*}
  \int_{K_{j}}|\nabla q_{j}|^2  \lesssim \sum_{i \in I, \atop z_i \in K_j} \bigl(\int_{\partial B_{r_i}(z_i)} |\partial_n v_i|^2 \bigr)^{\frac 1 2} \bigl(\int_{\partial B_{r_i}(z_i)}|q_{j}|^2 \bigr)^{\frac 1 2}.
 \end{align*}
By the definition of $v_i$ (see also \eqref{explicit.formula}), we rewrite the above inequality as
\begin{equation}
  \begin{aligned}\label{H.minus.6.bis}
  \int_{K_{j}}|\nabla q_{j}|^2  &\lesssim \sum_{i \in I, \atop z_i \in K_j} r_i^{-1} \bigl(\frac{X_i r_i}{r_i - X_i}\bigr) \bigl(\int_{\partial B_{r_i}(z_i)}|q_{j}|^2 \bigr)^{\frac 1 2}\\
  & \stackrel{\eqref{well.defined.cells}}{\leq}  \sum_{i \in I, \atop z_i \in K_j} r_i^{-1} X_i \bigl(\int_{\partial B_{r_i}(z_i)}|q_{j}|^2 \bigr)^{\frac 1 2} .
 \end{aligned}
 \end{equation}
 By the Trace embedding $L^2(\partial B_{r_i}(z_i)) \hookrightarrow H^1(B_{r_i}(z_i))$ we have
 \begin{align}\label{trace.embedding}
 \int_{\partial B_{r_i}(z_i)}|q_{j}|^2 \lesssim  r_i^{-1 }\bigl( \int_{B_{r_i}(z_i)}|q_{j}|^2  + r_i^2 \int_{B_{r_i}(z_i)} |\nabla q_j|^2 \bigr)
 \end{align}
 so that this, \eqref{H.minus.6.bis} and an application of Cauchy-Schwarz's inequality imply
    \begin{align}\label{H.minus.6}
  \int_{K_{j}}|\nabla q_{j}|^2  \lesssim \bigl( \sum_{i \in I, \atop z_i \in K_j} X_i^2 r_i^{-3} \bigr)^{\frac 1 2} \bigl( \sum_{i \in I, \atop z_i \in K_j} \int_{B_{r_i}(z_i)} (|q_{z}|^2  + r_i^2  |\nabla q_z|^2) \bigr)^{\frac 1 2}\\
  \stackrel{\eqref{contains.balls}-\eqref{well.defined.cells}}{\lesssim}  \bigl( \sum_{i \in I, \atop z_i \in K_j} X_i^2 r_i^{-3} \bigr)^{\frac 1 2} \bigl(\int_{K_j} (|q_{z}|^2  + \mathop{diam}(K_j)^2 |\nabla q_z|^2) \bigr)^{\frac 1 2}
 \end{align}
Since by \eqref{def.q.eps} the function $q_j$ has zero mean, we may apply Poincar\'e-Wirtinger's inequality to conclude  that
  \begin{align}\label{H.minus.7}
  \int_{K_{j}}|\nabla q_{j}|^2  \lesssim  \mathop{diam}(K_j)  \bigl( \sum_{i \in I, \atop z_i \in K_j} X_i^2 r_i^{-3} \bigr)^{\frac 1 2} \bigl(\int_{K_{j}}|\nabla q_{j}|^2 \bigr)^{\frac 1 2}.
 \end{align}
This establishes \eqref{H.minus.5.aux} and, in turn, concludes the proof of Lemma \ref{Kohn_Vogelius}.
\end{proof}

\bigskip

\begin{proof}[Proof of Lemma \ref{product.measure}] Without loss of generality we assume that $|A|=1$.
By the assumption (i)-(ii) on $(\Phi, \mathcal{R})$ we have that
\begin{align}
\E& \bigl[ \sum_{z \in \Phi(A)}  G( (z, \rho_z), \omega \backslash (z, \rho_z) ) \bigr]= e^{-\lambda} \sum_{n \geq 1} \frac{\lambda^n}{n!}\\
&\times  \sum_{i=1}^n \int_{(A \times \R_+)^n } \E \bigl[ G( (x_i, \rho_i), \omega \backslash \{ (x_i, \rho_i)\} )  \, | \, \Phi(A), \{\rho_z \}_{\Phi(A)} \bigr] f(\rho_1) \d\rho_1 \d x_1 \cdots  f(\rho_n) \d\rho_n \d x_n.
\end{align}
and, by symmetry,
\begin{align}
\E &\bigl[ \sum_{z \in \Phi(A) }   G( (z, \rho_z), \omega \backslash \{(z, \rho_z)\} ) \bigr] = \lambda e^{-\lambda} \sum_{n \geq 1} \frac{\lambda^{n-1}}{(n-1)!}  \\
&\times \int_{(A\times \R_+)^{n} } \E \bigl[  G( (x_1, \rho_1), \omega \backslash \{(x_1, \rho_1)\} )  \, | \, \Phi(A), \{\rho_z \}_{\Phi(A)} \bigr] f(\rho_1) \d\rho_1 \d x_1 \cdots  f(\rho_n) \d\rho_n \d x_n.
\end{align}
Appealing to Fubini's theorem and relabelling the elements $\{ (x_i, \rho_i)\}_{i=1}^n$, this implies
\begin{align}
\E &\bigl[ \sum_{z \in \Phi(A) } G((z, \rho_z), \omega \backslash \{(z, \rho_z)\} ) \bigr] = \lambda \int_{A \times \R_+}  \biggr(e^{-\lambda} \sum_{n \geq 0} \frac{\lambda^{n}}{n!}  \\
&\times \int_{(A\times \R_+)^{n} } \E \bigl[ G( (x, \rho), \omega)  \, | \, \Phi(A), \{\rho_z \}_{\Phi(A)} \bigr] f(\rho_1) \d\rho_1 \d x_1 \cdots  f(\rho_n) \d\rho_n \d x_n \biggr) \, f(\rho) \, \d x,
\end{align}
i.e.
\begin{align}
\E &\bigl[ \sum_{z \in \Phi(A) } G( (z, \rho_z), \omega \backslash \{(z, \rho_z)\} ) \bigr] = \lambda \int_{A} \E_\rho \biggl[ \E \bigl[ G((x, \rho) , \omega) \bigr] \biggr] \, \d x
\end{align}
Since $\Phi$ is stationary, the above inequality immediately implies Lemma \ref{product.measure}.
\end{proof}

\end{document}